\documentclass[a4paper,11pt]{amsart}
 \usepackage{amssymb,amsfonts,latexsym}
 \usepackage{graphics,verbatim}
 \usepackage{graphicx}
 
 \usepackage{upref}
 \usepackage{hyperref}
 \usepackage{xcolor}
 \usepackage{amsthm}
 
 \newtheorem{theorem}{Theorem}[section]
 \newtheorem{proposition}{Proposition}[section]
 \newtheorem{lemma}{Lemma}[section]
 
 \newtheorem{remark}{Remark}[section]

 \renewcommand{\(}{\left(}
 \renewcommand{\)}{\right)}
 \renewcommand{\[}{\left[}
 \renewcommand{\]}{\right]}
 \newcommand{\eps}{\varepsilon}
 
 \newcommand{\grad}{\ensuremath{\nabla}}
 
 \newcommand{\N}{\ensuremath{\mathbb{N}}}
 \newcommand{\R}{\ensuremath{\mathbb{R}}}
 
 \newcommand{\rn}{\ensuremath{\mathbb{R}^n}}

 \newcommand{\Real}{\mathbb{R}}

 \newcommand{\RPlus}{\Real^{+}}

 \newcommand{\abs}[1]{\left\vert#1\right\vert}

 \newcommand{\norm}[1]{\left\Vert#1\right\Vert}
 \newcommand{\inp}[1]{\langle#1\rangle}
 
 \newcommand{\be} {\begin{equation}}
 \newcommand{\ee} {\end{equation}}
 \newcommand{\bea} {\begin{eqnarray}}
 \newcommand{\eea} {\end{eqnarray}}
 \newcommand{\Bea} {\begin{eqnarray*}}
 	\newcommand{\Eea} {\end{eqnarray*}}

 \newcommand{\pa} {\partial}
 
 \newcommand{\al} {\alpha}
 \newcommand{\ba} {\beta}
 \newcommand{\de} {\delta}
 \newcommand{\ga} {\gamma}
 
 \newcommand{\Om} {\Omega}
 
 \newcommand{\De} {\Delta}
 
 \newcommand{\La} {\Lambda}

 \newcommand{\I}{\infty}
 \newcommand{\tf} {\tfrac}

 \newcommand{\f}{\frac}

 \newcommand{\ml}{\mathcal}
 \newcommand{\mb}{\mathbb}
 \newcommand{\ef}{\eqref}
 \newcommand{\nee}{\notag\ee}

 \newcommand{\pOm}{\partial\Om}

 \newcommand{\vst}{\vspace{10pt}}

 \catcode`\@=11
 
 \makeatletter \@addtoreset{equation}{section} \makeatother

\begin{document}

 	\title[On a singularly perturbed anisotropic elliptic system ]
 	{Concentration results for solutions of a  singlarly perturbed elliptic system with variable coefficients}

 	\author[A.K. Sahoo]{Alok Kumar Sahoo}
 	
 	\address{Dept. of Mathematics, IIT Hyderabad,\\ India-502285. }
 	
 	\email{aloksahoomath@gmail.com}

 	\author[B.B. Manna]{Bhakti Bhusan Manna}
 	\address{Dept. of Mathematics, IIT Hyderabad, \\India-502285.}
 	\email{bbmanna@iith.ac.in} 
 	
 	
 	\date{}
 	
 	\begin{abstract}
 		In this article we shall study the following elliptic system with coefficients:
 				\begin{equation}\notag
 		\left\{\begin{aligned}
 		-\eps^2\De u +c(x)u=b(x) \abs{v}^{q-1}v, &\text{ and } -\eps^2\De v +c(x)v=a(x) \abs{u}^{p-1}u &&\text{in } \Om\\
 		u>0, \ v>0 \text{ in } \Om, &\text{ and }\quad\f{\pa u}{\pa\nu} = 0 = \f{\pa v}{\pa\nu} &&\text{on }\pOm
 		\end{aligned}
 		\right.
 		\end{equation} 
 		where $\Om$ is a smooth bounded domain in $\rn, n\ge 3$. The coefficients $a(x), b(x)$ and $c(x)$ are positive bounded smooth functions. We shall study the existence of point concentrating solutions and discuss the role of the coefficients to determine the concentration profile of the solutions. We have also discussed some applications of our main theorem towards the existence of solutions concentrating on higher-dimensional orbits.
 	\end{abstract}
 	
 	\maketitle

 	
 	\section{Introduction}
 	
 	Consider the following singularly perturbed coupled elliptic system:
 	\begin{equation}
 	\label{I1}
 	\left\{\begin{aligned}
 	-\eps^2\De u +c(x)u=b(x) \abs{v}^{q-1}v, &\text{ and } -\eps^2\De v +c(x)v=a(x) \abs{u}^{p-1}u\\
 	u>0, \ v>0 \text{ in } \Om, &\,\,\,\text{ and }\,\,\,\f{\pa u}{\pa\nu} = 0 = \f{\pa v}{\pa\nu} \quad\text{ on }\pOm
 	\end{aligned}
 	\right.
 	\end{equation}
 	where $\Om$ is a smooth bounded domain in $\rn$, $\eps$ is a small positive parameter and $n\ge 3$. The exponents $p,q$ satisfy $p,q> 1$ and
 	\be
 	\f{1}{p+1}+\f{1}{q+1}>\f{n-2}{n}\label{rangpq}
 	\ee
 	i.e $p$ and $q$ lies below the critical hyperbola (for reference look at \cite{MR1177298}). The weights $a,b,$ and $c$ are smooth functions with \be \label{c-asum} K_1\leq a(x),b(x),c(x)\leq K_2 \text{\quad in\quad} \bar{\Om}  \ee for some real constants $K_1,K_2>0$.
 	
 	In this work, we shall study the point concentration behavior for some least energy solution of the above system. We will also describe the role of the coefficients to determine the location of the concentration. 
 	
 	For the scalar case, the equation with nonconstant weights appears in many cases where the authors want to study the existence of higher-dimensional concentrating solutions (see \cite{MR2608946, MR3250368, MR3259003, MR3229826, MR3620895, MR3032301}).  In \cite{MR1974510, MR2056434}, Ambrosetti, Malchiodi, and Ni studied the existence of spherical concentrating solutions.  Later in \cite{MR2608946, MR3250368, MR3259003}, the authors proved the existence of solutions concentrating on higher dimensional ($\mathbb{S}^1, \mathbb{S}^3$, and  $\mathbb{S}^7$) orbits. In these works (\cite{MR2608946, MR3250368, MR3259003}), the orbits of concentration produced by a reduction process (Hopf Fibration), which leads to an anisotropic problem in lower dimensions. Some more results related to anisotropic equations, showing the existence of solutions concentrating in higher dimensional orbits can be found in \cite{MR3192458, MR3620895, MR3032301}. 
 	
 	The singularly perturbed elliptic system, with Neumann boundary condition, was first studied by Avila and Yang in \cite{MR1978382}. They proved the existence of nontrivial positive solutions whose point of maximums approaches to a common point on the boundary as $\eps$ goes to 0. This result was generalized by Ramos and Yang for some general convex nonlinearities in \cite{MR2135746}. Then in \cite{MR2057542}, Pistoia and Ramos proved the concentration happens at the point of maximum of the mean curvature of the boundary. All these results are related to the existence of point concentrating solutions and the profile of the concentration. In \cite{MR3056708}, the authors studied the existence of radial solutions for a Neumann system with radial weights.
 	
 	To the best of our knowledge, this is the first study for the Neumann elliptic system with non-trivial weights. We have considered general weight functions, which help us to find many other higher dimensional layered solutions for different singularly perturbed systems. We shall show the concentration profile mainly depends upon the coefficients $a, b$ and $c$ and, in a special case, on the mean curvature of the boundary too. We define a function $\Lambda$ on the boundary of the domain,(see \ef{concfn} for more details)
 	\be \Lambda(x) =  \(\f{b(x)}{c(x)}\)^{ \f{p+1}{1-pq}}\(\f{a(x)}{c(x)}\)^{ \f{q+1}{1-pq}}   (c(x))^{1-\f{n}{2}} \label{concfn1}\nee
 	The main result in this paper is the following,
 	
 	\begin{theorem}\label{Thm1}
 		Under assumption \eqref{rangpq},\ef{c-asum} $\exists$ an $\eps_0$ such that for $0<\eps<\eps_0$ the equation \eqref{I1} has non constant positive 
 		solutions $u_\eps, \ v_\eps\in C^2(\Om)$. Moreover, both solutions concentrate for $\eps \to 0$ on a common point $P_\eps\in\pOm$, with $P_\eps$ satisfying:
 		\begin{itemize}
 			\item [(i)] $\lim_{\eps\to 0}P_\eps=\inf_{x\in\pOm} \Lambda(x) $ if $\La$ is not constant.
 			\item [(ii)] $\lim_{\eps\to 0}P_\eps=\sup_{x\in\pOm} H(x)\gamma(x)+\eta(x) $ if $\La$ is constant.
 		\end{itemize}
 		Where $H$ is the mean curvature of the boundary and $\gamma(x)=\f{C_1}{\sqrt{c(x)}}$, $\eta(x)=\f{C_2}{\sqrt{c(x)}}\f{\pa}{\pa_{x_n}}(C_3\ln a(x)+C_4\ln b(x)+C_5\ln c(x))$ are functions on $\pOm$. And the constants are given in Proposition \ref{bconst}.
 	\end{theorem}

 	Note that the concentration profile of the solution depends on the weight functions only unless the function $\La$ becomes constant. The following system in $\rn$ plays an important role in finding the asymptotic profile of the solutions. 
 	\be -\De u+u=\abs{v}^{q-1}v  \text{ and } -\De v+v=\abs{u}^{p-1}u, \qquad u, v >0\text{ in } \rn. \label{limeq}\ee

 	The energy functional $I_\eps(u,v): H^1(\Om)\times H^1(\Om) \to \R$ associated to equation \eqref{I1} is given by
 	\be I_\eps(u,v)=\int_\Om[\eps^2\langle\grad u,\grad v\rangle +c(x)uv-a(x) F(u)-b(x) G(v)]dx\label{enI1}\ee
 	with $F(s) = \frac 1{p+1} |s|^{p+1}$ and $G(s) = \frac 1{q+1} |s|^{q+1}$.
 	Lets $\ml{H}:=H^1(\Om)\times H^1(\Om)$ equipped with the norm $\norm{(u,v)}=\norm{u}+\norm{v}$. Note that under the assumption (\ref{rangpq}) one can have $q + 1 >2^* >p+1$ and $J_\eps$ may not be well defined over $\ml{H}$. In \cite{MR1978382} and \cite{MR1785681}  authors proved the existence result using Dual variation method and Fraction setting respectively.  But for a modified problem as in \cite{MR2135746}, we shall see that the solutions are uniformly bounded and hence the integrals are well defined (One can see from\cite{MR2342612} also). So for the moment,  we assume $2<p \le q<\f{N+2}{N-2}$. Then $J_\eps$
 	is well defined and belongs to $C^2(\ml{H},\R)$. Furthermore
 	\be DI_\eps(u,v)(\phi,\psi)=\langle u,\psi\rangle_\eps + \langle \phi,v\rangle_\eps -\int_\Om [a(x) f(u)\phi+b(x)g(v)\psi]dx\label{denI1}\ee
 	where $f(u)=\abs{u}^{p-1}u$, $g(v)=\abs{v}^{q-1}v$ and the quadratic term of the energy $\langle \cdot , \cdot\rangle_\eps$ is defined as 
 	\be \langle u, v\rangle_\eps=\int_\Om[\eps^2\langle\grad u,\grad v\rangle +c(x)uv]dx\label{ipI1}\ee\\
 	is positive definite (negative definite) on $E^+(E^-)$ where $E^\pm:=\{(\phi, \pm\phi):\phi\in H^1(\Om)\}$. Note $\ml{H}=E^+\oplus E^-$.
 	\par \bigskip
 	
 	The paper is organized as follows. In the second section, using of Benci-Rabinowitz's theorem, we have proved the existence of solutions $(u_\eps, v_\eps)$ for $\eps$ small enough. In the third section, we have described some behavior of the solutions and the solutions of the limit problem. In the fourth and fifth section, we have shown the upper and lower energy estimate of the energy functional giving the proof of main theorem. Finally, in section 6 we shall discuss some applications, as mentioned earlier. \medskip 
 	
 	\section{\textbf{ Existence of Mountain pass-type solution and known results}}	
 	
 In this section, our goal is to find solution to the problem $\ef{I1}$ from truncated system.	Without loss of generality, one can assume $q+1\ge p+1>2$, and $p+1<\f{2N}{N-2}$. Then for the sequence $j = 1, 2,3, . . .$ define
 	\be g_j(s) = 
 	\begin{cases}
 		A_j \abs{s}^{p-1}s + B_j&\text{for $ s\ge j $},\notag \\
 		\quad s^q&\text{for $ |s|\le j $},\notag \\
 		\tilde{A}_j\abs{s}^{p-1}s+\tilde{B}_j&\text{for $ s\le -j $}.\notag \\
 	\end{cases}
 	\notag\ee
 	The coefficients are 
 	\begin{align*}
 	A_j=\Big(\f{q}{p}+o(1)\Big)j^{q-p}=\tilde{A}_j \text{ and } B_j=\Big(\f{p-q}{p}+o(1)\Big)j^q=-\tilde{B}_j,
 	\end{align*}

 	chosen in such a way that the functions $g_j$ become $C^1$ (See \cite{MR2135746}). And we consider the modified problem 
 	 	\begin{equation}
 	\label{MI1}
 	\left\{\begin{aligned}
 	-\eps^2\De u +c(x)u=b(x) g_j(v), &\text{ and } -\eps^2\De v +c(x)v=a(x) f(u)&&\text{in } \Om\\
 	u>0, \ v>0 \text{ in } \Om, &\,\,\,\text{ and }\,\,\,\f{\pa u}{\pa\nu} = 0 = \f{\pa v}{\pa\nu} \quad\text{ on }\pOm
 	\end{aligned}
 	\right.
 	\end{equation}
 	The energy functional is
 	\be I_\eps(u,v)=\int_\Om[\eps^2\langle\grad u,\grad v\rangle +c(x)uv-a(x) F(u)-b(x) G_j(v)]dx\label{modeng}\ee
 Hence, the energy functional \ef{modeng} is $C^2$ over the Hilbert space $\ml{H}$.Furthermore, if $(u_{\eps_j},v_{\eps_j})$ solves the system \ef{MI1} and are uniformly bounded ,then $(u_{\eps_j},v_{\eps_j})$ solves the problem \ef{I1} for a large $j$. The proof of uniform boundedness is closely related to the proof given in \cite{MR2135746}. 
 
 We recall the following Liouville-type result from \cite{MR2135746}. Suppose $f_\I$ and $g_\I$ are $C^1$ functions satisfy, for some positive constants $c_1,c_2$ ,$\forall s\in \R$, the fillowing:
 	 \be \begin{aligned}
 	     & \text(a)\,\,\, c_1|s|^{q+1} \leq g_\I(s)s\leq c_2 |s|^{q+1}      \hspace{90mm}\\
 	    &  \text(b)\,\,\, g_\I(s)s\leq (q+1)G_\I(s) \\
 	   & \text(c)\,\,\, p\, g_\I(s)s\leq g^{\prime}_\I(s)s^2  \\
 	  \end{aligned} \nee
 Let $(u,v)$ be $C^2$ solution of $-\De u= g_\I(v), -\De v= f_\I(u)$ in $\R^n$ or with neumann data in  $\R^n_+$(up to rotation and translation). We say $(u,v)$ has finite index if $\exists R_0>0$ with property that for all $\phi\in H^1$ such that support of $\phi$ in $\complement B(0,R_0)$ with $2\int|\grad \phi|^2-f^{\prime}_\I(u)\phi^2-g^{\prime}_\I(v)\phi^2 \geq 0$.
 	\begin{proposition}[proposition 1.4 in \cite{MR2135746}]\label{p21}
 	 Let $g_\I\in C^1(\R)$ and $f_\I=C_0|s|^{p-1}s$, with $C_0>0$ , $p,q$ satisfies critical hyperbola condition and suppose $(u,v)$ has finite index ,  
 	 \be \begin{aligned}.
 	 	& \text(i)\,\,\, \text{if $g_\I$=0, then u=0}      \hspace{90mm}\\
 	  	& \text(ii)\,\,\, \text{if $g_\I$ satisfies $(a),(b),(c)$ then u=0=v}.
 	 \end{aligned} \nee
 	\end{proposition}
 
And we have 
 \begin{proposition}\label{unibond}
 	For any given sequence $\eps_j$, let $u_j,v_j$ be solutions to the problem \ef{MI1} with $\eps=\eps_j$. if there exists $k\in \N$ such that $m(u_j,v_j)\leq k$ for every $j$, then $\exists \  K>0$ such that \be \norm{u_j}_\I+ \norm{v_j}_\I\leq K\qquad \forall j.\nee 
 \end{proposition}
 \begin{proof} 
        Suppose $\norm{u_j}_\I+ \norm{v_j}_\I\to \I$. Define \be M_j =\sup_{x\in \bar{\Om}}\,\{\,max\,\{|u_j(x)|^{\f{1}{q+1}}\, ,\,|v_j(x)|^{\f{1}{p+1}}\} \}\nee
        Clearly $M_j\to \I$. Let $\lim_{j\to\I} \f{j}{M_j^{p+1}}=l\in[0,\I]$(upto a subsequence).\\
        $Case\,1.\,\,l>0$ \quad Now choose a sequence $\lambda_j\in \RPlus$ and $x_j\in \bar{\Om}$ such that $\lambda^2_jM^{pq-1}_j=1$  , $ M_j = \,\{\,max\,\{|u_j(x_j)|^{\f{1}{q+1}}\, ,\,|v_j(x_j)|^{\f{1}{p+1}}\} \}$ respectively.  Consider the blow-up scheme 
        \be  \tilde{u}_j(x)=\f{1}{M^{q+1}_j}u_j(\lambda_j\eps_j x+x_j) ,\quad \tilde{v}_j(x)=\f{1}{M^{p+1}_j}v_j(\lambda_j\eps_j x+x_j) \quad\text{in} \,\,\tilde{\Om}_j=\f{\Om-x_j}{\lambda_j\eps_j}\nee
        As $ \norm{\tilde{u}_j}_\I, \norm{\tilde{v}_j}_\I\leq 1$, so  $(\tilde{u}_j, \tilde{v}_j)$ converges in $C^2_{loc}$ to $(u,v)$, solves the limit problem $-\De u= g_\I(v),  -\De v= f_\I(u)$, where $f_\I(s)=a(x_0)|s|^p$ and
         \be g_\I(s) = b(x_0)
        \begin{cases}
        	\f{q}{p}l^{q-p} s^p + \f{p-q}{p}l^{q}&\text{for $ s\ge l $},\notag \\
        	\quad s^q&\text{for $ \abs{s} < l $},\notag \\
        	\f{q}{p}l^{q-p} s^p - \f{p-q}{p}l^{q}&\text{for $ s\le l $},\notag\\
        \end{cases}
        \notag\ee
        Since $\norm{ {u} }_\I, \norm{ {v} }_\I\leq 1$, one can easily check condition (a), (b) and (c) and Proposition \ref{p21} tells that $u=0=v$ contradicting the fact that one of $\tilde{u}_j(0), \ \tilde{v}_j(0)$ has to be one.\\ \\
        $Case\,2.\,\,l=0$ \quad Similar to $Case\,1.\,\, $  using Blow-up scheme, we will arrive at the limit problem whose solution gives contradiction as in \cite{MR2135746}.
 \end{proof}

 	\begin{lemma}
 		$I_\eps$ satisfies the Pales-Smale condition. 
 	\end{lemma}
 	\begin{proof}
 		For a Palais-Smale sequence $(u_n,v_n)$ of $I_\eps$ we have 
 		\be I'_\eps(u_n,v_n)(u_n,v_n)=2\inp{u_n,v_n}_\eps -\int_\Om a(x)f(u_n) u_n dx-\int_\Om b(x)g_j(v_n) v_n dx\label{Se1.1}
 		\ee
 		Noting that $ g_j(s)s \ge (p+1)G_j(s)$, and $f(s)s = (p+1) F(s)$,  for all $s$, we have 
 		
 		\begin{align}
 			2I_\eps(u_n,v_n)&-I'_\eps(u_n,v_n)(u_n,v_n) = \int_\Om a(x)(f(u_n)u_n-2F(u_n) )+\int_\Om b(x)(g_j(v_n)v_n-2G_j(v_n))\notag   \\ 
 			& \ge \int_\Omega a(x)\(\frac {p-1}{p+1}\) f(u_n)u_n  +  \int_\Omega b(x)\(\frac {p-1}{p+1}\) g_j(v_n)v_n \label{Se1.6}
 		\end{align}
 		and hence
 		\be \int_\Om\Big[f(u_n)u_n+g_j(v_n)v_n\Big]dx \le C_1+C_2\,\mu_n(\norm{u_n}_\eps+\norm{v_n}_\eps).\notag\ee
 		Now note that $a^{p+1}+b^{p+1}-a^pb-ab^p=(a-b)(a^p-b^p)\ge0$, and that there exists a constant $C\equiv C(j,A_j,B_j)$ such that $\int_\Om g_j(v_n)u_n dx \le C\int_\Om v_n^pu_n dx$. Then we can estimate
 		\begin{align}
 			\norm{u_n}^2_\eps+\norm{v_n}^2_\eps& \le I'_\eps(u_n,v_n)(v_n,u_n) +C \int_\Om\Big[a(x)u_n^pv_n+b(x)v_n^pu_n\Big]dx \notag\\
 			&\le \mu_n(\norm{u_n}_\eps+\norm{v_n}_\eps)+ C_1\int_\Om\Big[u_n^{p+1}+v_n^{p+1}\Big]dx\notag\\
 			&\le \mu_n(\norm{u_n}_\eps+\norm{v_n}_\eps)+ C_2\int_\Om\Big[f(u_n)u_n+g_j(v_n)v_n\Big]dx\notag\\
 			&\le C_3+C_4\,\mu_n(\norm{u_n}_\eps+\norm{v_n}_\eps)\label{Se1.2}
 		\end{align}
 		So we have $\norm{u_n}_\eps+\norm{v_n}_\eps\le C$.
 		Hence every Palais-Smale sequence $(u_n,v_n)$ is bounded. Then we have up to  subsequence $u_n\rightharpoonup u$ and $v_n\rightharpoonup v$ in $\ml{H}$, and clearly $I'_\eps(u,v)=0$. Now from \eqref{Se1.6} we have for large $j$ 
 		\begin{align} 
 			2I_\eps(u,v)&=\f{p-1}{p+1}\int_\Om a(x)(u^+)^{p+1} dx
 			+\f{q-1}{q+1}\int_{\Om\cap\{v\le j\}}b(x)(v^+)^{q+1}dx\notag\\
 			+ &\int_{\Om\cap\{v\ge j\}} b(x)\Big(\f{p-1}{p+1}A_j v^{p+1} + B_j v\Big)dx\ge 0\label{Se1.7}
 		\end{align}
 		Thanks to compact embedding,    
 		\be \int_{\Om} F(u_n-u)dx= \int_\Om F(u_n)dx - \int_\Om F(u)dx+o(1),\label{Se1.8}\ee and the same is true for $G_j(u_n-u)$. Hence we have for $(\phi,\psi)\in \ml{H}$
 		\be \abs{\int_\Om f(u_n-u)\phi dx - \int_\Om f(u_n)\phi dx + \int_\Om f(u)\phi dx}\le o(1)\norm{\phi},\label{Se1.9}\ee and the same for $g_j, \psi$.
 		\par \smallskip
 		Plugging $\bar{u}_n=u_n-u$ and $\bar{v}_n=v_n-v$ into \ef{Se1.1}  and from \ef{Se1.8}, \ef{Se1.9},
 		\begin{align}
 			I_\eps(\bar{u}_n,\bar{v}_n)\notag
 			=&\int_\Om\eps^2\Big[\inp{\bar{u}_n,\bar{v}_n}+ c(x)\bar{u}_n\bar{v}_n - a(x)F(\bar{u}) - b(x)G_j(\bar{v})\Big]dx\notag\\
 			=& I_\eps(u_n,v_n)-I_\eps(u,v)+o(1)\notag 
 		\end{align}
 		It is easy to see 
 		\be I'_\eps(\bar{u}_n,\bar{v}_n)(\phi,\psi)= I'_\eps(u_n,v_n)(\phi,\psi) - I'_\eps(u,v)(\phi,\psi)+o(1)=o(1)\notag\ee
 		
 		\be \norm{\bar{u}_n}^2_\eps + \norm{\bar{v}_n}^2_\eps=I'_\eps(\bar{u}_n,\bar{v}_n)(\bar{v}_n,\bar{u}_n)+
 		\int_\Om\Big[\f{\bar{u}_n^p\bar{v}_n}{(2\abs{x})^{1-\f{\al}{2}}}+\f{\bar{v}^p\bar{u}_n}{(2\abs{x})^{1-\f{\beta}{2}}}\Big]dx\to 0\notag\ee
 	\end{proof}

 	We consider the following orthogonal splitting of the space  $\ml{H}^1$.  
 	\be 
 	\ml{H}=\ml{H}_+\oplus\ml{H}_-, \text{ where } \ml{H}_{\pm}:=\{(\phi,\pm\phi):\phi\in H^1(\Om)\}\notag
 	\ee
 	
 	It is straight foreword to verify the conditions of Theorem 1.1, \cite{MR1804760} i.e.$I_\eps\leq 0 $ in $\ml{H}_-$ and $I_\eps\geq\rho>0$ in $\ml{H}_+\cap \pa B(0,r)$ for some small $r,\rho >0$. Furthermore, for a large $R(\eps)>0$ with $0<e=(e_1,e_2)\in\ml{H} $ such that \be\sup_{(\ml{H}_-\bigoplus \R^+e)\cap \pa B(0,R)}I_\eps\leq0 .\notag\ee  Then $I_\eps$ admits a critical point $(u_\eps,v_\eps)$ having morse index $m(u_\eps,v_\eps)\leq 1$ with energy $I_\eps(u_\eps,v_\eps)$ bounded between $\rho$ and $\sup_{(\ml{H}_-\bigoplus \R^+e)}I_\eps$. Let

 	\be 
 	\mathcal{N}_\eps=\big\{(u,v):(u,v)\in\ml{H}\backslash\{(0,0)\},I'_\eps(u,v)=0 \text{ and } I_\eps(u,v)\geq \rho \big\} \ , \notag
 	\ee
 	
 	As $\mathcal{N}$ is a nonempty set, lets define the minimal energy critical level \be c_\eps:= \inf_{(u ,v )\in\mathcal{N}_\eps}I_\eps(u ,v )\notag \ee
 	
 	The Palais-smale property of $I_\eps$ ensures the existence of  $(u_\eps,v_\eps)\in \mathcal{N}$ s.t. $I_\eps(u_\eps,v_\eps)=c_\eps$. Again as in Theorem 1.1  \cite{MR1804760}, We may assume $m(u_\eps,v_\eps)\leq 1$. Thanks to Hopf's lemma,  $(u_\eps,v_\eps)>0$. Again, relative Morse index gives that solutions are non-constants.

 	\par \medskip
 	
 	\section{\textbf{ Preliminary estimates}}
 	
 	Let $u_\eps,v_\eps\in\ml{H}$ be any ground-state solutions for system \eqref{MI1}. 
 	Then $u_\eps > 0$, $v_\eps > 0$. As $u_\eps,v_\eps\in\ml{H}$ uniformly bounded, from now on we shall denote $g_j$ by $g$ and $G_j$ by $G$. Let  $x_\eps \in \bar{\Om}$ be such that \be \max_{\bar{\Om}} u_\eps=u_\eps(x_\eps).\notag\ee
 	Let us fix a sequence $\eps_j$ in such a way that $x_j:=x_{\eps_j}\to x_0\in\bar{\Om}$ and $z_j\to z_0\in\pa\Om$ where $z_j\in\pa\Om$ such that
 	\be d_j:=dist(x_j,\pa\Om)=\abs{x_j-z_j}\notag\ee 
 	
 	We denote $u_j:=u_{\eps_j}$ and $v_j:=v_{\eps_j}$. The re-scaled solutions \be\bar{u}_j(x):=u_j(\eps_jx+x_j), \ \bar{v}_j(x):=v_j(\eps_jx+x_j), x\in\Om_j:=\tf{1}{\eps_j}(\Om-x_j)\label{Chv1}\ee
 	solve the system
 	
 	\begin{equation}
 	\label{RSE}
 	\left\{\begin{aligned}
 	-\De \bar{u}_j + c( \eps_jx+x_j)\bar{u}_j= b(\eps_jx+x_j)g(\bar{v}_j), &\text{ and }\\ 
 	-\De \bar{v}_j + c( \eps_jx+x_j) \bar{v}_j= a( \eps_jx+x_j)f(\bar{u}_j)   &\text{ in } \Om_j\vspace{0.2cm}\\
 	\bar{u}_j>0, \ \bar{v}_j>0 \text{ in } \Om_j, \text{ and } \f{\pa \bar{u}_j}{\pa\nu} = 0 = \f{\pa \bar{v}_j}{\pa\nu} &&\text{on }\partial \Om_j
 	\end{aligned}
 	\right.
 	\end{equation}\\
 	
 	Furthermore, the corresponding energy functional of \ref{RSE}  in Cartesian coordinates takes the form
 	\begin{align}
 		\label{enRE}&I_j(u_j,v_j)=\notag\\ 
 		&\int_{\Om_j}\langle\grad u_j,\grad v_j\rangle + c(x_j+\eps_j x)u_{j}v_j- a(x_j+\eps_j x)F(u_j)-b(x_j+\eps_j x)G(v_j)dx
 	\end{align}
 	where $F(s)= \int_{0}^{s}f(t)dt$ and $G(s)= \int_{0}^{s}g(t)dt$. According to proposition \ref{unibond} (see \cite{MR2057542}), we may assume $p=q$. The critical points of the functional \ef{enRE} are now the solutions of the system  \ef{RSE}. By $L^p$ and the Schauder estimate, we have the convergence in $H^1(\rn)$ and in $C^2_{loc}(\rn)$ to a nonzero solution of the limit system\\
 	\begin{equation}
 	\label{RLE}
 	\left\{\begin{aligned}
 	-\De u +c(x_0) u=  b(x_0)g(v), &\text{ and } 
 	-\De v + c(x_0) v =  a(x_0) f(u) &&\text{in } \ml{U}\\
 	     u>0, \ v>0 \text{ in } \ml{U}, &\text{ and } \quad\f{\pa u}{\pa \nu} = 0 = \f{\pa v}{\pa \nu}  \quad\text{on }\partial \ml{U} 
 	\end{aligned}
 	\right.
 	\end{equation}\\
 	where $\ml{U}$ is the open set $\ml{U}=\{x\in\R^n : \inp{x,n(x_0)}<\rho_0\}$, where 
 	\be \rho_0=\lim_{j\to \I}\rho_j, \ \rho_j:=d_j/\eps_j\notag
 	\ee
 	The corresponding energy functional is \\
 	\be I_{\abs{x_0}}(u,v)=\int_\ml{U}\Big[\langle\grad u,\grad v\rangle + c( x_0) uv -a(x_0)F(u) -
 	b(x_0)G(v) \Big]dx\label{enRLE}\ee\\
 Since $\rho_j \to 0 $ ((a), lemma \ref{UMPlemma}), so without loss of generality we can take $\ml{U}=\rn_+$. \\ \\
 		\begin{remark}\label{remark}
 			\end{remark}
 		Suppose $I''_\I(u,v)(\phi u,\phi u)(\phi u,\phi u) \geq 0$ and $ I''_\I(u,v)(\phi v,\phi v)(\phi v,\phi v) \geq 0$ for all test function in $\ml{U}.$  Multiplying both sides of $-\De u +c(x_0) u=  b(x_0) v^q$ and $-\De v + c(x_0) v =  a(x_0) u^p$ with $\phi^2u $ and $\phi^2v $ respectively we get  
  
 	

 	\begin{align*}
 		\int 2(u^2+v^2)|\grad \phi|^2\geq &\int \phi^2[(u^2+v^2)(a(x_0)f^\prime (u)+b(x_0)g^\prime (v))-    (a(x_0)f(u)v+b(x_0)g(v)u)]\\
 		&\geq \int \phi^2(a(x_0)f(u)v+b(x_0)g(v)u)
 		\notag 
 	\end{align*}
 	Last step follows from the fact $p,q>1$ and  $u^2+v^2\geq2uv$. \\ \\
 
 	 Now we shall discuss some results similar to the scalar case as in Boyen and Park \cite{MR2180862}.
 	\begin{proposition}
 		Let $u,v$ be any positive radially symmetric solution to the problem \ef{RLE},then 
 		\begin{align*}
 			(i).\,& \int_{\rn_{+}} \f{\pa u}{\pa z_n}  \f{\pa v}{\pa z_n} z_n dz=\f{2}{n+1}\int_{\rn_{+}}\langle \grad u ,  \grad v\rangle z_n dz\\
 			(ii).\,&\int_{\rn_{+}}[a(x_0)\Big(\f{1}{2}f(u)u-F(u)\Big)+ b(x_0)\Big(\f{1}{2}g(v)v-G(v)\Big)]z_n\\&=\int_{\rn_{+}}\Big[\langle\grad u,\grad v\rangle + c( x_0) uv -a(x_0)F(u) -
 			b(x_0)G(v) \Big]z_n dz + \f{1}{2}\int_{\pa\rn_+}uvd\sigma\\
 			(iii).\,&\int_{\rn_{+}}\Big[\langle\grad u,\grad v\rangle + c( x_0) uv -a(x_0)F(u) -
 			b(x_0)G(v) \Big]z_n dz=2\int_{\rn_{+}}\f{\pa u}{\pa z_n}  \f{\pa v}{\pa z_n} z_n dz\\
 			(iv).\,&\int_{\rn_{+}}\f{\pa u}{\pa z_i}  \f{\pa v}{\pa z_i} z_n=\f{1}{n+1}\int_{\rn_{+}}\langle \grad u ,  \grad v\rangle z_n dz\,,\quad  \,i=1,2,...,n-1.
 		\end{align*} 
 		\label{estimate}
 	\end{proposition} 	
 	
 	\begin{proof}
 		We start with the polar coordinate system as in \cite{MR1115095}. Since $u,v$ are radial, using integration in polar coordinates $(i),(iv)$  follows. The second estimation can be seen by integration by-parts formula. Multiplying $z^2_n\f{\pa u}{\pa z_n}$ and $z^2_n\f{\pa v}{\pa z_n}$ with $-\De v + c(x_0) v =  a(x_0) u^p$ and $-\De u +c(x_0) u=  b(x_0) v^q$ respectively, followed by integration gives the required estimate.
 	\end{proof}

 	Towards the goal $i.e.$ to characterize the location of spikes, we consider the following change of variable.
 	\begin{align}\label{LimChVr1}
 		U(x)=\(\f{b(x_0)}{c(x_0)}\)^{\al_1}\(\f{a(x_0)}{c(x_0)}\)^{\ba_1}u(\f{x}{ \sqrt{c( x_0)}})\\
 		V(x)=\(\f{b(x_0)}{c(x_0)}\)^{\al_2}\(\f{a(x_0)}{c(x_0)}\)^{\ba_2}v(\f{x}{ \sqrt{c( x_0)}}) \label{LimChVr}
 	\end{align}
 	where $\al_1=\f{ 1}{(pq-1)}$ , $\al_2=\f{ p}{(pq-1)}$ , $\ba_1=\f{ q}{(pq-1)} $ and $\ba_2=\f{ 1}{(pq-1)} $. 
 	
 	Under the above transformations, system \eqref{RLE} modified into following system.
 	\begin{equation}
 	\label{CRLE}
 	\left\{\begin{aligned}
 	-\De U +U=V^q, &\text{ and } -\De V +V=U^p &&\text{in } \R^n_{+}\\
 	U>0, \ U>0 \text{ in } R^n_{+}, &\text{ and } \f{\pa U}{\pa\nu} = 0 = \f{\pa V}{\pa\nu} &&\text{on } \pa R^n_{+}.
 	\end{aligned}
 	\right.
 	\end{equation} We denote the energy 
 	\be I_{\infty}(u,v)=\int_{\R^{n}_+}\Big[\langle\grad u,\grad v\rangle +  uv - F(u) -
 	G(v) \Big]dx \label{EnLp}\ee and 
 	\be \Lambda(x) =  \(\f{b(x)}{c(x)}\)^{-\al_1-\al_2}\(\f{a(x)}{c(x)}\)^{-\ba_1-\ba_2}   (c(x))^{1-\f{n}{2}} \label{concfn}\ee

 	Under the above change of variable, One can easily see the energy takes the form
 	
 	
 	
 	\begin{align}
 		I_{|x_0|}(u,v)&= \(\f{b(x_0)}{c(x_0)}\)^{-\al_1-\al_2}\(\f{a(x_0)}{c(x_0)}\)^{-\ba_1-\ba_2}   (c(x_0))^{1-\f{n}{2}} I_{\infty}(U,V)\label{enCRLE}\\ 
 		&=  \Lambda(x_0) \,\, I_{\infty}(U,V)\notag
 	\end{align}
 	
 	It is well known (see \cite{MR1617988, MR1785681}) that all strong positive solutions of \eqref{CRLE} are radially symmetric, and there exists a 
 	ground state radially symmetric solution $U,V$ of (\ref{CRLE}) such that $U(x)=U(|x|)$ and $V(x)=V(|x|)$, satisfying the decay estimates
 	\be 
 	|D^\alpha U(x)|, \ |D^\alpha V(x)|\le C\exp(-\delta|x|),\label{DcyLe}\ee for some $c,\delta>0$ and for all $|\alpha|\le 2$.\\
 	 \par


 	\begin{lemma}\label{UMPlemma}
 		Let  $u_j, v_js $ be the solutions of \ref{I1}. Then 
 		\begin{itemize}
 			\item [(a)]There exists $ C ,\eps_0>0 $ such that for any $ 0<\eps \leq \eps_0 $ ,$ u_j $ has a maximum point $x_j$ satisfying 
 			\be dist(x_j, \pa\Om )\leq C \eps \notag\ee 
 			\item [(b)]$x_j \in \pa\Om $ for j sufficiently large.
 			\item [(c)] $x_j$ is also the unique max point of $v_j$ for j sufficiently large.
 		\end{itemize}
 	\end{lemma}
 	\begin{proof}
 		Using the method of contradiction, result has been proved in \cite{MR2135746}. Here we will only sketch the difference arising due to positive weight coefficients. First we will prove $(b),(c)$ under assumption of $(a)$ i.e. $x_j\to x_0\in \pa \Om$ and argument based on Theorem 2.1 of \cite{MR2135746}. Using reflection through hyperplane and unique maximum point of limit problem with change of variable \ef{LimChVr1} \ef{LimChVr}, one can see  \be \f{dist(x_j, \pa\Om )}{\eps_j}\to 0   \text{  \quad as }j\to \I\notag\ee 
 		
 		From Blow-up argument \ef{Chv1}, $(\bar{u}_j, \bar{v}_j)\to(u,v)$ in $C^2_{loc}$, where $(u,v)$ solves the limit problem \ef{RLE} and we have seen $(u,v)$ has exponential decay at infinity. 
 		Already we observed that (see Remark \ef{remark}) for a test function, either  $I''_\I(u,v)(\phi u,\phi u)(\phi u,\phi u) < 0$ or $ I''_\I(u,v)(\phi v,\phi v)(\phi v,\phi v) < 0$  must hold. So we may assume that, there exist a test function $\phi_1$ with support in $B(0,R_0)$ and $I''_\I(u,v)(\phi_1 u,\phi_1 u)(\phi_1 u,\phi_1 u) < 0$. Again, $C^2_{loc}$ convergence allows to see that $I''_j(\bar{u}_j,\bar{v}_j)(\phi_1 u,\phi_1 u)(\phi_1 u,\phi_1 u) < 0$ for large enough $j$. Since $m(u_j,v_j)\leq1$, $I''_j(\bar{u}_j,\bar{v}_j)(h_j,h_j)( h_j,h_j) \geq 0$ for all $h_j\in H^1(\Om_j)$ with $h_j=0$  in $ B(0,R_0)$.\\ \\
 		
 	Since limit solution decays exponentially,	given $\delta>0,\ \exists R>0$ such that
 	\be \int_{\ml{U}\cap(B(0,2R)\backslash B(0,R))}a(x_0)f(u)v+b(x_0)g(v)u<\delta \label{decay4fg}\ee\\
 	so, From the fact $C^2_{loc}$ convergence, For a large j 
 	
 	\be 
 	\label{decay2}\int_{ \Om_j\cap(B(0,2R)\backslash B(0,R))}a(\eps_j x+x_j)f(\bar{u}_j)\bar{v}_j+b(\eps_j x+x_j)g(\bar{v}_j)\bar{u}_j<\delta
 	\ee\\
 	 Fix a smooth cut-off function $\psi$ such that $\psi=0$ in $B(0,R)$ and $\psi=1$ in $\complement B(0,2R)$. So  $I''_j(\bar{u}_j,\bar{v}_j)(h_j,h_j)( h_j,h_j) \geq 0$ for $h_j=\psi \bar{u}_j$ as well as $\psi \bar{v}_j$ . Now remark \ref{remark} reads 
 	\be 
 	\label{decay3}
 	\int_{\Om_j\cap \complement B(0,2R)}  \(  a(\eps_jx+x_j)f(\bar{u}_j)\bar{v}_j+ b(\eps_jx+x_j)g(\bar{v}_j)\bar{u}_j\)<\delta
 	\ee \\
 	
 	Hence together \ef{decay2} and \ef{decay3} proves that for any $\delta>0$ there exist $R>0$(large enough) such that  
 	\be \label{fdecay}
 	\int_{\Om_j\cap \complement B(0,R)}  \(  a(\eps_jx+x_j)f(\bar{u}_j)\bar{v}_j+ b(\eps_jx+x_j)g(\bar{v}_j)\bar{u}_j\)<\delta
 	\ee \\

 	Let $y_j$ be any sequence of maximum points of $v_j$ in $\bar{\Om}$. First we claim that there exist a constant $C>0$	such that for $j$ very large, \be dist(x_j,y_j)\leq C\eps_j \nee
 	Define \be\tilde{u}_j(x):=u_j(\eps_jx+y_j), \ \tilde{v}_j(x):=v_j(\eps_jx+y_j), x\in\tilde{\Om}_j:=\tf{1}{\eps_j}(\Om-y_j) \nee
 	Since $\tilde{u}_j\to \tilde{u}, \tilde{v_j}\to \tilde{v}$ as in blow-up scheme defined for $\bar{u}_j,\bar{v}_j$ and $\tilde{v}(0)\ne0$, inequality \ef{fdecay} reads as
 	\be \label{ffdecay}
 	\int_{\tilde{\Om}_j\cap \complement B(0,R)}  \(  a(\eps_jx+y_j)f(\tilde{u}_j)\tilde{v}_j+ b(\eps_jx+y_j)g(\tilde{v}_j)\tilde{u}_j\)<\delta
 	\ee \\
 	
 	Denote $\eps_jz_j:=y_j-x_j$, if $|z_j|\to \I$ then both inequalities \ef{fdecay} and \ef{ffdecay} with \ef{c-asum} concludes
 	\be  
 	\int_{\Om_j }  \(   f(\bar{u}_j)\bar{v}_j+  g(\bar{v}_j)\bar{u}_j\)<2\delta
 	\ee \\	
 	Since $\delta$ is arbitrary, it contradicts the fact $\bar{v}(0)\ne0$.	This proves the claim.
 	Remaining estimations follow directly by lifting (see \ef{LimChVr1} and \ef{LimChVr} ) ODE condition and non degenerate maximum points from  $U,V$, radial solutions to the problem \ef{CRLE} to limit solution $u,v$.\\ 
 	Note that same proof \ef{decay4fg}-\ef{fdecay} can be carried to achieve 
 	\be \label{decayuv}
 	\int_{{\Om}_j\cap \complement B(0,R)}  \(  a(\eps_jx+y_j) \bar{u}^2_j+b(\eps_jx+y_j)\bar{v}^2_j\)<\delta
 	\ee  
 	
 	Similarly, again by contradiction argument we can observed that for any $\delta>0$ there exists $R>0, j_0\in N$ such that  \be |\bar{u}_j|+|\bar{v}_j|\leq \delta \, \quad,  x\in {\Om}_j\cap \complement B(0,R) \text{\,\,\,for all\,\,\,} j>j_0 \label{pointdecay}\ee.\\
 	
 	To see the proof, let there exist a $\rho>0$ and $z_n\in \Om_j$ such that $|z_n|\to \I$ and $|\bar{u}_j(z_j)|>\rho$ for all $j$. Define $\tilde{\Om}_j:=\Om_j-z_j$ and $\tilde{x}_j=\eps_jz_j+x_j \in \Om$.\\
 	Now set $\tilde{u}_j(x):=\bar{u}_j(x+z_j)=u_j(\eps_jx+\bar{x}_j)\,,\,\tilde{v}_j(x):=\bar{v}_j(x+ z_j)=v_j(\eps_jx+\bar{x}_j)$.\\
 	From the way of choosing the sequence $z_j$, we have $\tilde{u}_j(0)=u_j(\bar{x}_j)=\bar{u}_j(z_j)>\rho$ for all $j$.
 	Now \ef{fdecay}	reads as
 	\be  
 	\int_{\tilde{\Om}_j\cap \complement B(0,R)}  \(  a(\eps_jx+\bar{x}_j)f(\tilde{u}_j)\tilde{v}_j+ b(\eps_jx+\bar{x}_j)g(\tilde{v}_j)\tilde{u}_j\)<\delta
 	\ee 
 	\be  
 \Rightarrow	\int_{{\Om}_j\cap \complement B(z_j,R)}  \(  a(\eps_jx+{x}_j)f(\bar{u}_j)\bar{v}_j+ b(\eps_jx+{x}_j)g(\bar{v}_j)\bar{u}_j\)<\delta
 	\nee \\
 	
 	Since $z_j\to \I$, we have 	\be  
 	\int_{\Om_j }  \(   f(\bar{u}_j)\bar{v}_j+  g(\bar{v}_j)\bar{u}_j\)<2\delta
 	\nee \\
 	Which imply $C^1_{loc}$ limit of $\bar{u}_j,\bar{v}_j$ i.e. $u=v=0$. Which is a contradiction to the existence of the sequence $z_j$. This proves \ef{pointdecay}.
 	
 	Now, as $p,q>1$ from \ef{decayuv} and \ef{pointdecay}, for all $j>j_0$ we have 
 	\be \label{FFFdecay}
 	\int_{\Om_j\cap \complement B(0,R)}  \(  a(\eps_jx+x_j)f(\bar{u}_j)\bar{u}_j+ b(\eps_jx+x_j)g(\bar{v}_j)\bar{v}_j\)<\delta
 	\ee \\
 	
 	The estimation $(a)$ proved on basis of contradiction argument. Suppose there is no such $C$ such that $ dist(x_j, \pa\Om )\leq C \eps $ holds. let $\tilde{u}_j(x):=u_j(\eps_jx+x_j),  \tilde{v}_j(x):=v_j(\eps_jx+x_j), x\in\tilde{\Om}_j:={\eps_j}^{-1}(\Om-x_j) $ be the blow-up scheme with limit problem in $\R^n$, $u,v$ be the limit solution and both are readily symmetric with respect to origin. The decay property of $\tilde{u}_j,\tilde{v}_j$ i.e.\ef{FFFdecay} leads to the fact $I_j(\tilde{u}_j,\tilde{v}_j) + o(1)=I_{|x_0|}(u,v)$. Note that the energy $I_{|x_0|}(u,v)$ is over $\R^n$. The radial symmetry of limit solution $(u,v)$ gives,
 	\begin{align*} I_j(\tilde{u}_j,\tilde{v}_j) + o(1)&=I_{|x_0|}(u,v) \\
 		&=2\int_{ \R^{n}_+ }\Big[\langle\grad u,\grad v\rangle + c( x_0) uv -a(x_0)F(u) -
 		b(x_0)G(v) \Big]dx\\
 		&>\int_{ \R^{n}_+ }\Big[\langle\grad u,\grad v\rangle + c( x_0) uv -a(x_0)F(u) -
 		b(x_0)G(v) \Big]dx
 	\end{align*}\\
 	Now we will follow Theorem 3.1 and Theorem 3.5 in \cite{MR2135746} with Appendix A to get a contradiction.  The key idea of these theorems is about finding maximum points along an energy curve $\al(t)=I(t(u,v)+(1-t)(\phi,-\phi)$ , $\phi \in H^1$. The maximum point $t=1$ exists and the existence is proved by finding another curve $\ba(t)$ which intersect $\al(t)$ at critical points.  
 \end{proof}

 	\par \medskip 
 	\section{Upper Energy Estimation}
 	
 	
 	
 	\begin{theorem}\label{Lem.UpEs}
 		Let $(u_\eps,v_\eps)$ be a minimal energy solution of system \eqref{I1} with the corresponding functional \eqref{enI1}. Then we have 
 		\be 
 		c_{\eps_j}\leq \eps_j^n \Big( \Lambda(x_0)\,I_{\infty}(u,v)-\eps_j[(n-1)H(x_0)\ga +\eta] + o(\eps_j)\Big)\label{UpEnEs}
 		\ee 
 		for $j$ sufficiently large.
 	\end{theorem}

 	let $x_j\in \pa\Om$ is the point of maximum of $(u_j,v_j)$ the solution of \ref{I1}, with limit $x_0\in\pa\Om$. Without loss of generality we can assume $x_0=0$ and the inner normal at $x_0$ to $\pOm$ directed towards positive $x_n$-axis. \\
 	
 	Then there exists neighborhood $B_{\de_1}(0)$ of $0$ in $\rn$, $\tilde{B}_{\de_2}(0)$ of $0$ in $\R^{n-1}$ and a map $\psi: \tilde{B}_{\de_2}(0)\to \R$ such that $\pOm\cap B_{\de_1}(0)$ is the graph of $\psi$. Now define a map $\Phi: \tilde{B}_{\de_2}(0)\times \R\to \rn$ by
 	\begin{align}\label{BdMap}
 		\Phi_j(y)=	\left\{\begin{aligned}
 			&y_j-y_n\f{\pa\psi}{\pa x_j}(y') &&\text{ for } j=1,2, \dots, n-1\\
 			&y_n+\psi(y') &&\text{ for } j=n
 		\end{aligned}
 		\right.
 	\end{align}	
 	Since $\grad\psi(0) = 0$, we note that $D\Phi(0)=Id$ and hence $\Phi$ is locally invertible at $0$ i.e. there exists balls $B_{\de}(0)$ , $B_{\de'}(0)$ and a map $\Psi: B_{\de'}(0)\to B_{\de}(0)$ such that $\Psi=\Phi^{-1}$ and $D\Psi(0)=D\Phi(0)^{-1}$.\\
 	
 	Now let us take $\Om_j:=\f{\Om-x_0}{\eps_j}$ and for $x\in\Om_j$ define $y=\eps_j x+x_0\in\Om$ and  $w=\Psi(y)=\Psi(\eps_j x+x_0)\in\rn$ for $y\in$ domain of $\Psi$ and $z=\f{w}{\eps_j}=\f{\Psi(y)}{\eps_j}=\f{\Psi(\eps_j x+x_0)}{\eps_j}$. Then note that $y=\Phi(\eps_j z)$\\
 	
 	let $\chi\in C^\I(\rn,\R) $ is radially symmetric, $\chi=1$ in $  B(0, \de)$ and $\chi = 0 $ in $\rn \backslash B(0, 2\de)$. For a small $\de>0$, such that $B = B_+(0,2\de)$ subset of domain of $\Phi$. Define for $x\in\Om_j \cap B_+(0,2\de)$\\
 	\be 
 	u_{j}(x) = u\Big(\f{\Psi(\eps_{j}x+x_0)}{\eps_j}\Big) \chi(\Psi(\eps_j x+x_0) \nee
 	\be
 	v_{j}(x) = v\Big(\f{\Psi(\eps_{j}x+x_0)}{\eps_j}\Big)\chi(\Psi(\eps_j x+x_0)
 	\nee
 	Where $(u,v)$ solves the limit problem \ref{RLE} over $\rn_{+}$.

 	It has been shown in appendix (A) that $u_j,v_j\in H^1(\Om_j)$ and they solves following PDE over  $\Om_j$
 	\begin{equation}
 	\label{TFE}
 	\left\{\begin{aligned}
 	&-\De u_j + c(x_0) u_j =  b(x_0) v_j^q+\mu_j(x),  \\& -\De  v_j + c(x_0) v_j  = a(x_0)u_j^p +\nu_j(x) \quad \text{ in } \Om_j\\
 	&u>0, \ v>0 \text{ in } \Om_j, \,\,\,\text{ and }\f{\pa u}{\pa\nu} = 0 = \f{\pa v}{\pa\nu}  \,\,\,\text{ on }\pOm_j
 	\end{aligned}
 	\right.
 	\end{equation}
 	
 	Where $\mu_j$ and $\nu_j$ are given in \ref{MUj} and \ref{NUj}.\\
 	
 	Let $H(p)$ denotes the mean curvature of $\pOm$ at the point $p$, then we have $\De \psi(0)=(n-1)H(0)$ and 
 	\be |D(\Phi(y))|= 1-(n-1)H(0)y_n + O(|y|^2). \ee

 	We denote the energy $I_j(u_j,v_j)$ 
 		\begin{align*}
 	=&\int_{\Om_j}\Big[\langle\grad u_j,\grad v_j\rangle + c(x_0+\eps_j x)u_{j}v_j- a(x_0+\eps_j x)F(u_j)-
 	b(x_0+\eps_j x)G(v_j)\Big]dx\\
 =& \int_{\Om_j} \Big[\f{1}{2}b(x_0)g(v_j)v_j-G(v_j)b(x_0+\eps_j x)\Big]+\f{1}{2}\Big[\mu_{j}v_j+\nu_{j}u_j\Big]\\& +
 		\Big[\f{1}{2}a(x_0)f(u_j)u_j-F(u_j)a(x_0+\eps_j x)\Big]+\Big[ u_{j}v_jc(\eps_j x+x_0)-c(x_0)u_{j}v_j \Big]\\
 	\end{align*}

 	
  Let $(u,v)$ and $(U,V)$ are solutions to the problems \ef{RLE} and \ef{CRLE} respectively. We fix some notations. In proposition \ref{bconst}, we will establish a more compact form of the boundary functions. 
       \begin{align*}  
 			& \ga = \ga(f)+\ga(g)+\xi-\tau,\quad &&\eta = \eta(f)+ \eta(g)- \Theta \\
 			&\ga(f):=a(x_0)\int_{\R^{n}_{+}}\Big(\f{1}{2}f(u)u- F(u)\Big)x_{n}dx \,\,\,
 			&&\eta(f):=\int_{\R^{n}_{+}} F(u)<a'(x_0),x>dx\\
 			& \ga(g):=b(x_0)\int_{\rn_{+}} \Big(\f{1}{2}g(v)v- G(v)\Big)x_{n}dx \,\,
 			&& \eta(g):=\int_{\rn_{+}} G(v)<b'(x_0),x>dx  \\
 			& \Theta :=\int_{\R^{n}_{+}}uv<c'(x_0),x> dx && \quad \xi:=\int_{\rn_{+}} \f{\pa u}{\pa z_k}(z) \f{\pa v}{\pa z_k}(z)z_n dz, \\
 			&\tau:=\f{1}{2}\int_{\pa\rn_+}uvd\sigma
 		\end{align*} . 
 \begin{proposition}\label{bconst}
 		\begin{align*}
 	&(i)\,\,\ga=\f{5}{n+1}\f{\La(x_0)}{\sqrt{c(x_0)}}\int_{\rn_{+}}\langle \grad U ,  \grad V\rangle z_ndz\\
 	&(ii)\,\,\eta = \f{\La(x_0)}{\sqrt{c(x_0)}}\[ \f{\pa_n a(x_0)}{a(x_0)}\int_{\R^{n}_{+}}F(U)x_n + \f{\pa_n b(x_0)}{b(x_0)}\int_{\R^{n}_{+}}G(V)  x_n -  \f{\pa_n c(x_0)}{c(x_0)}\int_{\R^{n}_{+}}UV  x_n \] 
 	\end{align*}
 \end{proposition}

\begin{proof}

 The proof follows easily from Proposition \ref{estimate} along with change of variable  \ef{LimChVr1}-\ef{LimChVr}. 
 \begin{align*} 
  \ga& = \ga(f)+\ga(g)+\xi-\tau \\
  &=\int_{\rn_{+}}\Big[\langle\grad u,\grad v\rangle + c(x_0) uv -a(x_0)F(u) -
  b(x_0)G(v) \Big]z_n  + \f{1}{2}\int_{\pa\rn_+}uvd\sigma+\xi-\tau\\
  &=2\int_{\rn_{+}}\f{\pa u}{\pa z_n}  \f{\pa v}{\pa z_n} z_n dz+ \f{1}{2}\int_{\pa\rn_+}uvd\sigma+\xi-\tau\\
  &=\f{4}{n+1}\int_{\rn_{+}}\langle \grad u ,  \grad v\rangle z_n + \f{1}{2}\int_{\pa\rn_+}uvd\sigma+\f{1}{n+1}\int_{\rn_{+}}\langle \grad u ,  \grad v\rangle z_n dz-\f{1}{2}\int_{\pa\rn_+}uvd\sigma\\
  &=\f{5}{n+1}\int_{\rn_{+}}\langle \grad u ,  \grad v\rangle z_ndz\\
  	&=\f{5}{n+1}\f{\La(x_0)}{\sqrt{c(x_0)}}\int_{\rn_{+}}\langle \grad U ,  \grad V\rangle z_ndz
  \end{align*}

Note that, radial symmetry of $u,v$ implies that both $u(x^{\prime},x_n),v(x^{\prime},x_n)$ are even in $x^{\prime}$. Thus, under symmetry the integral can be reduced to a simpler form.
\begin{align*}
	\eta &= \eta(f)+ \eta(g)- \Theta\\
	&=\int_{\R^{n}_{+}} F(u)\langle a'(x_0),x \rangle +   G(v)\langle b'(x_0),x \rangle - uv\langle c'(x_0),x\rangle \\
	&= \pa_n a(x_0)\int_{\R^{n}_{+}}F(u)  x_n +  \pa_n b(x_0)\int_{\R^{n}_{+}}G(v)  x_n -  \pa_n c(x_0)\int_{\R^{n}_{+}}uv  x_n \\
	&=\f{\La(x_0)}{\sqrt{c(x_0)}}\[ \f{\pa_n a(x_0)}{a(x_0)}\int_{\R^{n}_{+}}F(U)x_n + \f{\pa_n b(x_0)}{b(x_0)}\int_{\R^{n}_{+}}G(V)  x_n -  \f{\pa_n c(x_0)}{c(x_0)}\int_{\R^{n}_{+}}UV  x_n \] 
	\end{align*}
\end{proof}	
 In lemma \ref{Lem.UpEs5} we prove the asymptotic expansion of the energy $I_j(u_j,v_j)$ and the proof follows from Lemma \ref{Lem.UpEs2} and Lemma \ref{Lem.UpEs3}.
 	\begin{lemma}\label{Lem.UpEs5}
 	\be
 	I_j(u_j,v_j) = I_{|x_0|}(u,v)-\eps_j[(n-1)H(x_0)\ga +\eta] + o(\eps_j)\label{UpEnEs1}\nee 
  \end{lemma}
 
 	
 	\begin{lemma}\label{Lem.UpEs2}
 		\begin{align*}
 			\int_{\Om_j}&\f{1}{2}a(x_0)f(u_j)u_j-F(u_j)a(x_0+\eps_j x) dx \\
 			& = \int_{\R^{n}_{+}} 	a(x_0)\Big(\f{1}{2}f(u)u-F(u)\Big) dx - \eps_{j}\Big((n-1)H(x_0)\ga(f) +\eta(f)\Big) + o(\eps_j).  
 		\end{align*}
 	\end{lemma}

 	\begin{proof}
 		For $\eps$ very small, we have
 		\begin{align*}
 			&\int_{\Om_j}\f{1}{2}a(x_0)f(u_j)u_j-F(u_j)a(x_0+\eps_j x) dx\\
 			= &  \int_{\Om_j}\f{1}{2}a(x_0)f(u_j)u_j-F(u_j)\[  a(x_0)+\eps_j< a^{\prime}(x_0),x> + o(\eps_j)\]dx\\
 			=& \int_{\Om_j}\(\[\f{1}{2}f(u_j)u_j-F(u_j)\]  a(x_0)-F(u_j)\eps_j< a^{\prime}(x_0),x> + o(\eps_j)\)dx .\\
 		\end{align*}
 		Consider the first part of above integral. Note that for the change of variable $z=\f{\Psi(\eps_j x+x_0)}{\eps_j}$ equivalent to  $x=\f{\Phi(\eps_j z)-x_0}{\eps_j}$ one has \be dx=[1-\eps_j(n-1)H(x_0)z_n+ O(|\eps_j z|^2)]dz\nee set $ K(x_0):=(n-1)H(x_0)$ and we have\\
 		\begin{align*}
 			& \int_{\Om_j}\[\f{1}{2}f(u_j(x))u_j(x)-F(u_j(x))\]dx \\ 
 			& \f{(p-1)}{2(p+1)}\int_{B_+(0,2\de/\eps_j)}\Big( u(z) \chi(\eps_j z)\Big)^{p+1} (1-\eps_j(n-1)H(x_0)z_n + O(|\eps_j z|^2))dz\\
 			=&\f{(p-1)}{2(p+1)} \int_{B_+(0,\f{\de}{\eps_j})}\Big( u (z)\Big)^{p+1} (1-(n-1)H(x_0)z_n\eps_j + O(|\eps_j z|^2))dz\,\,  \\+
 			&\f{(p-1)}{2(p+1)} \int_{B_+(0,\f{2\de}{\eps_j})\backslash B_+(0,\f{\de}{\eps_j})}\Big( u(z) \chi(\eps_j z)\Big)^{p+1} (1-(n-1)H(x_0)z_n\eps_j )dz
 		\end{align*}
 		As $(u,v)$ has exponential decay. By straight-froward calculation, we have
 		\begin{align}
 			&\int_{\Om_j}\[\f{1}{2}f(u_j(x))u_j(x)-F(u_j(x))\]dx\notag\\ 
 			=&\f{(p-1)}{2(p+1)}\int_{\R^{n}_{+}}u^{p+1} (z) [1-\eps_j(n-1)H(x_0)z_n ] dz + o(\eps_j)\label{LEes1}
 		\end{align}
 		
 		\vst
 		
 		Similarly, for the second term
 		\begin{align}
 			&   \int_{\Om_j} u^{p+1}_j(x) <a'(x_0),x>dx\notag\\
 			=&  \int_{B_+(0,2\de/\eps_j)}\Big( u(z) \chi(\eps_j z)\Big)^{p+1}(1-\eps_j K(x_0)z_n )<a'(x_0),\f{\Phi(\eps_j z)-x_0}{\eps_j}>dz\notag\\
 			=&  \int_{B_+(0,\f{\de}{\eps_j})}\Big( u(z)\Big)^{p+1}(1- K(z_0)z_n\eps_j )<a'(x_0),\eps_j z )>dz +o(\eps_j)  \notag\\
 			=&  \int_{B_+(0,\f{\de}{\eps_j})}\Big( u(z)\Big)^{p+1}(1-\eps_jK(x_0)z_n  )<a'(x_0),\eps_j x>dz + o(\eps_j)\notag\\
 			=&  \eps_j \int_{\R^{n}_{+}}\Big( u(z)\Big)^{p+1}<a'(x_0),x>dz+o(\eps_j) \label{LEes2}
 		\end{align}
 		Hence lemma follows from \ref{LEes1} and \ref{LEes2}.
 	\end{proof}


 	\begin{lemma}\label{Lem.UpEs3}
 		\begin{align*}
 			(1) &\int_{\Om_j}\Big[\f{1}{2}b(x_0) g(v_j)v_j - G(v_j)b(x_0+\eps_j x) \Big]dx \\
 			&=\int_{\R^{n}_{+}} b(x_0)\Big(\f{1}{2}g(v)v-G(v)\Big) dx - \eps_{j}\Big((n-1)H(x_0)\ga(g) +\eta(g)\Big) + o(\eps_j)  \\
 			(2) &\int_{\Om_j}\Big[ u_{j}v_jc(\eps_j x+x_0)-c(x_0)u_{j}v_j \Big]dx = \eps_j  \Theta + o(\eps_j) \\ 
 			(3) &\,\,\f{1}{2}	\int_{\Om_j}\Big[\mu_{j}v_j+\nu_{j}u_j\Big] dx = -\eps_j(n-1)H(x_0)(\xi-\tau)+ o(\eps_j)
 		\end{align*}
 		Where the constants $\xi=\int_{\rn_{+}} \f{\pa u}{\pa z_k}(z) \f{\pa v}{\pa z_k}(z)z_n dz $ for $1\le k\le n-1$ and $\tau=\f{1}{2}\int_{\pa\rn_+}uvd\sigma$
 	\end{lemma}
 	Proof of statement $(3)$  follows from \ef{MUj} and \ef{NUj}.

 	\begin{lemma}\label{lem4.4}
 		\be \sup_{E^{-}\bigoplus\R^+(u_j,v_j)} I_j = I_j(u_j,v_j) + o(\eps_j) \nee
 	\end{lemma}
 	\begin{proof}
 		 We will show that the same argument [\cite{MR2057542}, lemma 3.3] can be adopted for our purpose with proper care of positive coefficients. The assumption \ef{c-asum} will take care of all the issues due to coefficients in energy estimations. The proof based on the exponential decay of limit solution and integration by parts  i.e.\ef{TFE} with decay \ef{MUj}, \ef{NUj}. The whole idea is to study asymptotic behaviour of maximum points $(s_j,t_j)$ of the energy  \be \chi_j(s,t):=I_j\( s(u_j,v_j)+ t(\phi_j,-\phi_j)\).\nee
 		 Moreover we can prove $\chi_j(s_j,t_j)=\chi_j(1,0)+o(\eps_j)$

 	\end{proof}
 	\begin{proof}[\textbf{Proof of Theorem \ref{Lem.UpEs}}]
 		From the definition of least energy solution(see the way of construction), we have \\
 		\be 
 		I_j(\tilde{u_j},\tilde{v_j})\leq  \sup_{E^{-}\bigoplus\R^+(u_j,v_j)} I_j  \,\,\,\,\, 
 		\nee\\
 		Where  \be\tilde{u}_j(x):=u_{\eps_j}(\eps_jx+x_0), \ \tilde{v}_j(x):=v_{\eps_j}(\eps_jx+x_0), x\in\Om_j:=\tf{1}{\eps_j}(\Om-x_0)\nee with least energy solution  $(u_{\eps_j}, v_{\eps_j}) $ to the problem \ef{I1}.
 		It is easy to see, \\
 		\be 
 		c_{\eps}= \eps^{n} I_j(\tilde{u_j},\tilde{v_j}). 
 		\nee \\
 		Hence the theorem follows.\\
 	\end{proof}\medskip

 	\section{Lower Energy estimation} \medskip

 	\begin{theorem}\label{Lem.LEs} 
 		Let $(u_\eps,v_\eps)$ be a minimal energy solution of system \eqref{I1} with the corresponding functional \eqref{enI1}. Then we have 
 		\be 
 		c_{\eps_j}\geq \eps_j^n \Big( \Lambda(x_j)\,I_{\infty}(u,v)-\eps_j[(n-1)H(x_j)\ga +\eta] + o(\eps_j)\Big)\notag
 		\ee 
 		for $\eps$ sufficiently small.
 	\end{theorem}
 	
 	As in the previous lemma, we can take a particular coordinate system such that $x_j=0$ and the inner normal at $x_j$ to $\pOm$ directed towards the positive $x_n$ axis. And we define $\psi^j$, $\Phi^j$ and $\Psi^j$ in a similar way such that 
 	\begin{itemize}
 		\item [I.] $\grad \psi^j(0)=0$
 		\item [II.] $D\Phi^j(0)=Id$ and
 		\item [III.] $(\Psi^j)^{-1}=\Phi^j$
 	\end{itemize}
 	
 	Define, 
 	\be 
 	\bar{u}_{j}(z) = \tilde{u}_j\Big(\f{\Phi^j(\eps_{j}z)-x_j}{\eps_j}\Big) \chi(\Phi^j(\eps_j z)-x_j) \nee
 	\be
 	\bar{v}_{j}(x) =\tilde{v}_j \Big(\f{\Phi^j(\eps_j z)-x_j}{\eps_j}\Big) \chi(\Phi^j(\eps_j z)-x_j)\quad ,
 	z\in \R^{n}_+.
 	\nee
 	
 	Where $(\tilde{u},\tilde{v})$ solves the problem \ref{RSE}. Similarly as in Proposition 5.1 of \cite{MR1978382} we get positive constants $c, \theta$ such that 
 	\be \tilde{u}_j(x)\le ce^{-\theta \abs{x}} \text{ and } \tilde{v}_j(x)\le ce^{-\theta \abs{x}}\label{tuvdecay}\ee
 	
 	Now from Lemma \ref{APLemma3B} we see that $\bar{u}_j,\bar{v}_j\in H^1(\R^{n}_+)$ and they solve following PDE over  $ \R^{n}_+$  with neumann boundary condition.

 	\begin{equation}
 	\label{RE3}
 	\left\{\begin{aligned}
 	-\De \bar{u}_j(z) +c(\Phi^j(\eps_{j}z)) \tilde{u}_j(x)&=b(\Phi^j(\eps_{j}z))g(\tilde{v}_j) + \bar{\mu}_j(z) &&\text{ and } \\
 	-\De \bar{v}_j(z) +c(\Phi^j(\eps_{j}z))\bar{v}_j(x) &= a(\Phi^j(\eps_{j}z))f(\tilde{u}_j) +\bar{\nu}_j(z)&&
 	\end{aligned}
 	\right.
 	\end{equation}
 	
 	where $x=\f{\Phi^j(\eps_j z)-x_j}{\eps_j}$ and $\bar{\mu}_j(z)$ , $\bar{\nu}_j(z)$ is given in Lemma\ref{APLemma3B}.
 	
   Note that
 	
 	\begin{align*}
 		& I_{|x_j|}(\bar{u}_j,\bar{v}_j) = \int_{\rn_+}\Big[\langle\grad \bar{u}_j,\grad \bar{v}_j\rangle +c(x_j)\bar{u}_j\bar{v}_j  - a(x_j)F(\bar{u}_j) -
 		b(x_j)G(\bar{v}_j)\Big]dx\notag\\
 		=&\int_{\rn_+}b(\Phi^j(\eps_{j}x))\Big(\f{1}{2}  g(\bar{v}_j)\bar{v}_j-G(\bar{v}_j) \Big)+  a(\Phi^j(\eps_{j}x))\Big( \f{1}{2} f(\bar{u}_j)\bar{u}_j- F(\bar{u}_j)\Big)\notag\\
 		&+\f{1}{2}\Big( \bar{\mu}_j\bar{v}_j+ \bar{u}_j\bar{\nu}_j \Big) +\big(c(x_j)-c(\Phi^j(\eps_{j}z))\big) \bar{u}_j\bar{v}_j \notag\\
 		&-\big(a(x_j)-a(\Phi^j(\eps_{j}z))\big) F(\bar{u}_j)-
 		\big(b(x_j)-b(\Phi^j(\eps_{j}z))\big)G(\bar{v}_j)+o(\eps_j) \label{LEest4}
 	\end{align*}

 	\begin{lemma}\label{Lem.LEs1} 
 		\begin{align*}
 			&\int_{\rn_+}\Big[\f{1}{2}g(\bar{v}_j)\bar{v}_jb(\Phi^{j}(\eps_j z)) 
 			-b(\Phi^{j}(\eps_j z))G(\bar{v}_j)\Big]dz \\
 			& =\int_{\Om_j} \(\f{1}{2} g(\tilde{v}_j)\tilde{v}_j-G(\tilde{v}_j)\) b((\eps_j x+x_j))dx\\&+   \eps_j (n-1)H(x_j) b(x_j)\int_{\rn_{+}}\Big(\f{1}{2}g(v)v-G(v)\Big)z_n dz + o(\eps_j)\\  
 		\end{align*}
 	\end{lemma}
 	
 	\begin{proof}
 		\begin{align*}
 			&\int_{\R^{n}_+}\Big[\f{1}{2 }b(\Phi^{j}(\eps_j z))g(\bar{v}_j)\bar{v}_j
 			-b(\Phi^{j}(\eps_j z))G(\bar{v}_j) \Big]dz\\
 			&= \Big[\f{1}{2}-\f{1}{q+1}\Big]\int_{\R^{n}_+}b(\Phi^{j}(\eps_j z))\bar{v}^{q+1}_j dz \\
 			&=\Big[\f{1}{2}-\f{1}{q+1}\Big]\int_{\R^{n}_+}b(\Phi^{j}(\eps_j z)) \tilde{v}^{q+1}_j\Big(\f{\Phi^j(\eps_j z)-x_j}{\eps_j}\Big)\chi^{q+1}(\Phi^j(\eps_j z)-x_j) dx\\
 		\end{align*}
 		Under the change of variable $\Phi^j(\eps_j z)-x_j=\eps_j x, \quad dz=(1+y_nK(x_j)+o(|y|)dx, \quad K(x_j):=(n-1)H(x_j)$ we get
 		
 		\begin{align*}
 			&= \int_{\Om_j \cap B(0,\f{2\de}{\eps_j})}b(\eps_j x+x_j)\tilde{v}^{q+1}_j(x)\chi^{q+1}(\eps_j x) [1+\eps_j x_nK(x_j)+O(|\eps_j x|^2)]dy\\ 
 			&= \int_{\Om_j \cap B(0,\f{\de}{\eps_j})} b(\eps_j x+x_j)\tilde{v}^{q+1}_j(x) [1+\eps_j x_nK(x_j)+O(|\eps_j x|^2)]dx\\
 			&+ \eps^{-n} \int_{\Om_j \cap [B(0,\f{2\de}{\eps_j})\backslash B(0,\f{\de}{\eps_j})]}b(\eps_j x+x_j)\tilde{v}^{q+1}_j(x)\chi^{q+1}(\eps_j x) [1+\eps_j x_nK(x_j) ]dx \\ 
 			&= \int_{\Om_j}b(\eps_j x+x_j)\tilde{v}^{q+1}_j(x)[1+\eps_j x_nK(x_j)+O(|\eps_j x|^2)]dx \\ 
 			&- \int_{\Om_j\cap( b(0,\f{\de}{\eps_j}))^c} b(\eps_j x+x_j)\tilde{v}^{q+1}_j(x) [1+\eps_j x_nK(x_j)+O(|\eps_j x|^2)]dx \\ 
 			&+ \int_{\Om_j\cap [b(0,\f{2\de}{\eps_j})\backslash b(0,\f{\de}{\eps_j})]}b(\eps_j x+x_j)\tilde{v}^{q+1}_j(x)\chi^{q+1}(\eps_j x) [1+\eps_j x_nK(x_j)+o(\eps_j)]dx \\
 			=& \int_{\Om_j} b( \eps_jx+x_j) \tilde{v}^{q+1}_j (x) dx+\eps_j(n-1)H(x_j) \int_{\Om_j} b( \eps_jx+x_j)\tilde{v}^{q+1}_j (x) x_n  +o(\eps_j)\\
 		\end{align*}
 		Considering the second term
 		\begin{align*}
 			& \int_{\Om_j}  b( \eps_jx+x_j)\tilde{v}^{q+1}_j (x) x_n dx \\
 			&=  \int_{\Om_j\cap b(0,\f{\de}{\eps_j})}  b( \eps_jx+x_j)\tilde{v}^{q+1}_j (x)\chi(x\eps_j )x_n dx + o(\eps_j)\\
 			&= \int_{ \R^{n}_+\cap b(0,\f{\lambda}{\eps_j})}b(\Phi^{j}(\eps_j z))\tilde{v}^{q+1}_j \Big(\f{\Phi^{j}(\eps_j z)-z_j}{\eps_j}\Big)  [\Phi^{j}_n(\eps_j z)-(z_j)_n] |D\Phi^{j}(\eps_j z)|dz  \\
 		\end{align*}
 		We shall use the following expansions:
 		\begin{align*}
 			&|D\Phi^{j}(\eps_j z)|=[1-K(x_j)z_n\eps_j + o(\eps_j)]\\
 			&b(\Phi^{j}(\eps_j x)) = b(x_j)+ \eps_j <B^{\prime}(x_j),z>+o(\eps_j)\\
 			&[\Phi^{j}_n(\eps_j z)-(x_j)_n]=z_n\eps_j+o(\eps_j).\\
 		\end{align*}
 		
 		\begin{align*}
 			&= \int_{ \R^{n}_+\cap b(0,\f{\lambda}{\eps_j})} b(x_j)\tilde{v}^{q+1}_j \Big(\f{\Phi^{j}(\eps_j z)-x_j}{\eps_j}\Big)  z_ndz + o(\eps_j)\\
 			&= \int_{ \R^{n}_+\cap b(0,\f{\lambda}{\eps_j})}b(x_j) \tilde{v}^{q+1}_j \Big(\f{\Phi^{j}(\eps_j z)-x_j}{\eps_j}\Big)\chi{(\Phi^{j}(\eps_j z)} x_ndx + o(\eps_j)\\
 			&= \int_{ \R^{n}_+ }b(x_j)v^{q+1}(x) z_ndx + o(\eps_j)
 			= \int_{ \R^{n}_+ }b(x_j)v^{q+1}(x) z_ndx + o(\eps_j)
 		\end{align*}	
 		
 	\end{proof}
 	
 	In a similar way, we have following results.
 	\begin{lemma}\label{Lem.LEs2} 
 		\begin{align*}
 			&\int_{\R^{n}_+}\Big[\f{1}{2}f(\bar{u}_j)\bar{u}_ja(\Phi^{j}(\eps_j z)) 
 			-a(\Phi^{j}(\eps_j z))F(\bar{u}_j)\Big]dz \\
 			& =\int_{\Om_j} \(\f{1}{2}-\f{1}{p+1}\)\tilde{u}^{p+1}_j a((\eps_j z+x_j))dz\\&+   \eps_j K(x_j) a(x_j)\int_{\R^{n}_{+}}  \(\f{1}{2}-\f{1}{p+1}\)u^{p+1} z_{n} dz + o(\eps_j)\\   
 		\end{align*}
 	\end{lemma} 
 	
 	\begin{lemma}\label{Lem.LEs3} 
 		\begin{align*}
 			(i) & \int_{\rn_+}\bar{u}_j\bar{v}_j\Big[  c(x_j)-c(\Phi^{j}(\eps_j z)) \Big]dz = -\eps_j\int_{\R^{n}_+} uv<\grad c(x_j), z>   dz+o(\eps_j)\\
 			(ii)& \int_{\rn_+}G(\bar{v}_j)\Big[  b(\Phi^{j}(\eps_j z)) - b(x_j)\Big]dz = \eps_j\int_{\rn_{+}} G(v)<b'(x_j),z>\Big]dz\\
 			(iii)&\int_{\rn_+}F(\bar{u}_j)\Big[  a(\Phi^{j}(\eps_j z)) - a(x_j)\Big]dz = \eps_j\int_{\rn_{+}} F(u)<a'(x_j),z>\Big]dz
 		\end{align*}
 	\end{lemma}
 	And from lemma\ref{APLemma4B} and Lemma\ref{APLemma5B} we get:\\
 	\begin{lemma}
 		\be
 		\int_{\rn_+}\f{1}{2}\Big[\bar{\mu}_{j}\bar{v}_j+\bar{\nu}_{j}\bar{u}_j\Big] dx = \eps_jK(x_j)(\xi-\tau)+ o(\eps_j)
 		\nee 
 		where $\xi:=\int_{\rn_{+}} \f{\pa u}{\pa z_k}(z) \f{\pa v}{\pa z_k}(z)z_n dz $ for $1\le k\le n-1$ and $\tau:=\f{1}{2}\int_{\pa\rn_+}uvd\sigma$
 	\end{lemma}
 	 Considering all above estimations, The energy $I_{|x_j|}(\bar{u}_j,\bar{v}_j)$ reads as
 	\begin{align*}
 		& I_{|x_j|}(\bar{u}_j,\bar{v}_j) \\
 		= &I_j(\tilde{u}_j,\tilde{v}_j)+\eps_j[K(x_j)(\ga(f)+\ga(g)+\xi-\tau)+(\eta(f)+\eta(g)-\Theta)]+o(\eps_j)\\
 		=&I_j(\tilde{u}_j,\tilde{v}_j)+\eps_j[ K(x_j)\ga+\eta] +o(\eps_j).\label{Les.0}
 	\end{align*}\\
 	At the end of this section, We recall lemma 4.1 from \cite{MR2057542} to conclude the theorem.
 	\begin{lemma}\label{Lem.LEs4} 
 		\be
 		I_{|x_j|}(u,v) \leq I_{|x_j|}(\bar{u}_j,\bar{v}_j) + o(\eps_j).
 		\ee 
 	\end{lemma}
 	\begin{proof}
 		It can be shown that, for any $e=(e_1,e_2)$ such that $e_i\in H^1(\R^n_+)$ and $e_1\ne-e_2$, \be I_{|x_j|}(u,v)\leq \sup_{\ml{H}_-\bigoplus \R^+e} I_{|x_j|}\nee
 		where $\ml{H}_- :=\{(\phi, -\phi):\phi\in H^1( \R^n_+)\}$. The proof follows by similar argument as in Lemma \ref{lem4.4} on the function $I_{|x_j|}$. \\
 	\end{proof}
 	\begin{proof}[ \textbf{Proof of Lemma\ref{Lem.LEs}}]  
 		
 		\quad   From Lemma\ref{Lem.LEs4},   
 		
 		\begin{align*}
 			& I_{|x_j|}(u,v) \leq  I_j(\tilde{u}_j,\tilde{v}_j)+\eps_j[ K(x_j)\ga+\eta] +o(\eps_j)\\ \\
 			& I_j(\tilde{u}_j,\tilde{v}_j) \geq  I_{|x_j|}(u,v)- \eps_j[ K(x_j)\ga+\eta] +o(\eps_j)\\  \\
 			& c_{\eps}\geq \eps^n [I_{|x_j|}(u,v)- \eps_j[ K(x_j)\ga+\eta] +o(\eps_j)]\\ 
 		\end{align*}
 		Hence we have lower estimate.\\
 	\end{proof}


 	\begin{proof}[\textbf{Proof of Theorem \ref{Thm1}}]
 		
 		\par \smallskip \noindent  \\ \\
 		{\it   Case (I)}: Suppose $\La$ is non constant. Since $x_j$ converges to $x_0$ we get from lemma\ref{Lem.UpEs} and lemma\ref{Lem.LEs} that the point of concentration is at the minimum point of the function $\La$.\\ \\
 		{\it   Case (II)}: If $\La$ is  constant then from the first order approximation of the energy $c_{\eps_j}$ and using the fact that $x_j\to x_0$, we get the point of concentration $x_0$, which minimizes the function $H(x)\ga(x)+\eta(x)$ in $\pOm$.
 	\end{proof}

 	\section{Applications: Concentration on higher dimensional orbits}
 	
 	In this section we shall show some applications of the previous result. 
 	
 	First let us consider the following equation: 
 		\begin{equation}
 	\label{ApE1}
 	\left\{\begin{aligned}
 	-\eps^2\De u +u=\abs{v}^{q-1}v, &\text{ and } -\eps^2\De v +v=\abs{u}^{p-1}u &&\text{in } A\\
 	u>0, \ v>0 \text{ in } \Om, &\text{ and }\f{\pa u}{\pa\nu} = 0 = \f{\pa v}{\pa\nu} &&\text{on }\partial A
 	\end{aligned}
 	\right.
 	\end{equation}
 	where $A=\{x\in\rn:a<\abs{x}<b\}$, with $0<a<b$.
 	Recall the only spheres which have group structure are $S^0, S^1, S^3, S^7, S^{15}$ (Hurwitz, 1898). And using the group structure one has the following classical Hopf fibration:
 	\begin{align*}
 		&S^0\hookrightarrow S^1\to \mathbb{RP}^1
 		&&S^1\hookrightarrow S^3\to \mathbb{CP}^1\equiv S^2\\
 		& S^3\hookrightarrow S^7\to \mathbb{HP}^1\equiv S^4
 		&&S^7\hookrightarrow S^{15}\to \mathbb{OP}^1\equiv S^8
 	\end{align*}
 	where $\mathbb{RP}, \mathbb{CP}, \mathbb{HP}, \mathbb{OP}$ are real, complex, quaternionic and octonionic projective spaces respectively. If $\pi$ is the corresponding Hopf maps in the above then $\pi$ is Harmonic Morphism, i.e. (J. C. Wood \cite{MR2250232})
 	\begin{align*}
 		&\De_{S^3}\longrightarrow \De_{S^2}\\ 
 		&\De_{S^7}\longrightarrow \De_{S^4}\\
 		&\De_{S^{15}}\longrightarrow \De_{S^8}\\
 	\end{align*} 
 	And one can easily determine the map $\pi$ reduces the system \ref{ApE1} to the following 
 	
 	\begin{equation}
 	\label{ApEC1}
 	\left\{\begin{aligned}
 	-\eps^2\De u +\f{1}{2\abs{x}}u=\f{1}{2\abs{x}}v^q, & \quad -\eps^2\De v +\f{1}{2\abs{x}}v=\f{1}{2\abs{x}}u^p \quad\text{in } \Om\\
 	u>0, \ v>0 \text{ in } \Om, &\text{ and }\quad\f{\pa u}{\pa\nu} = 0 = \f{\pa v}{\pa\nu} \qquad\text{on }\pOm
 	\end{aligned}
 	\right.
 	\end{equation}
 	where $\Om=(\f{a^2}{2}, \f{b^2}{2})\times S^{m-1}$, $m=3, 5, 9$ (for details see\cite{MR3250368}). 
 	And we have the theorem: 
 	
 		\begin{theorem}\label{ThmA1}
 		\begin{itemize} 
 			\item [(i)] for $n=4$ and $\f{1}{p+1}+\f{1}{q+1}>1/3$ the equation \ref{ApE1} has a solution which concentrates on a $S^1$ orbit in the inner boundary.
 			\item [(ii)] for $n=8$ and $\f{1}{p+1}+\f{1}{q+1}>3/5$ the equation \ref{ApE1} has a solution which concentrates on a $S^3$ orbit in the inner boundary.
 			\item [(iii)] for $n=16$ and $\f{1}{p+1}+\f{1}{q+1}>7/9$ the equation \ref{ApE1} has a solution which concentrates on a $S^7$ orbit in the inner boundary.
 		\end{itemize}
 	\end{theorem}
 	
 	\begin{proof}
 		\par \smallskip \noindent  \\
 		We observe that the equation \ref{ApEC1} is an anisotropic system with coefficients  $a(x)=b(x)=c(x)=\f{1}{2\abs{x}}$ which gives $\La(x)=(2\abs{x})^{n/2-1}$ for $n=3, 5, 9$. As $\La(x)$  attains infimum on the inner boundary of the annulus we have from the main result the point of concentration for the reduced problem is on the inner boundary. And hence the corresponding solution for the original problem concentrates on the corresponding 1, 3, and 7-dimensional spheres lying on the inner boundary of $A$ respectively.
 	\end{proof}

 	Now consider the following equation with weights:
 	\begin{equation}
 	\left\{\begin{aligned}\label{appln1}
 	-\eps^2\De u +u=\abs{x}^\beta v^q, &\text{ and } -\eps^2\De v +v=\abs{x}^\al u^p &&\text{in } A\\
 	u>0, \ v>0 \text{ in } \Om, &\text{ and }\f{\pa u}{\pa\nu} = 0 = \f{\pa v}{\pa\nu} &&\text{on }\partial A
 	\end{aligned}
 	\right.
 	\end{equation}
 	where $A =\{x\in \R^4 : 0< a< |x| <b\}$, $\alpha$, $\beta$ any real number. The exponents $p,q$ satisfy $p,q> 1$ and
 	\be
 	\f{1}{p+1}+\f{1}{q+1}>\f{1}{3}\label{appln2}
 	\ee
 	
 	

 	We look for the solutions $(u_\eps, v_\eps)$ which are invariant under $S^1$ action i.e. in the space:
 	\begin{align}
 		H_{\sharp}^1(A)=\{u\in H^1(A) : u(T_\tau(z))=u(z), \forall \tau \in [0,2\pi)\}.\label{appln3}
 	\end{align}
 	where $T_\tau$ is the following fixed point free one parameter group action on $A$
 	\be T_\tau(z)=z(r,t,\theta_1+\tau,\theta_2+\tau)\label{appln4}\ee for $\tau\in[0,2\pi)$. And the co-ordinate system we consider here is $A=I\times_r \mb{S}^3$.
 	Using the $S^1$ action we get the following reduced equation (similarly as in \cite{MR2608946})
 	
 	\begin{equation}
 	\left\{\begin{aligned}\label{appln5}
 	-\eps^2\De u +\f{u}{2\abs{x}}=\f{v^q}{(2\abs{x})^{1-\f{\beta}{2}}}, &\text{ and } 
 	-\eps^2\De v +\f{v}{2\abs{x}}=\f{u^p}{(2\abs{x})^{1-\f{\al}{2}}} &&\text{in } \Om\\
 	u>0, \ v>0 \text{ in } \Om, &\text{ and }\qquad \f{\pa u}{\pa \nu} = 0 = \f{\pa v}{\pa \nu} &&\text{on }\partial \Om
 	\end{aligned}
 	\right.
 	\end{equation}
 	where $\Om=\{x \in\R^3:\f{a^2}{2}<|x|<\f{b^2}{2}\}$
 	\par \medskip
 	Note that the equation \ref{appln5} is the anisotropic problem \ref{I1} with $a(x)=(2\abs{x})^{1-\f{\al}{2}}$, $b(x)=(2\abs{x})^{1-\f{\beta}{2}}$ and $c(x)=(2\abs{x})^{-1}$. And hence we get $\La(x)=\abs{x}^{\f{pq-1-\al(q+1)-\beta(p+1)}{pq-1}}.$\\

 	Again using the Theorem\ref{Thm1} we hve :
 	\begin{theorem}\label{ThmA2}
 		Under assumptions \ref{appln2}
 		there is an $\eps_0$ such that for all $0<\eps<\eps_0$ the equation \eqref{appln1} has non constant positive 
 		solutions $u_\eps, \ v_\eps\in C^1(\bar{A})$. Moreover, both solutions concentrate for $\eps \to 0$ on a common $\mb{S}^1$-orbit $S(r_\eps)$, where $r_\eps$ denotes the radius of the circular orbit, and satisfies:
 		\begin{itemize}
 			\item [(i)] $r_\eps \to b$ \qquad \qquad  for\ \ $2\al(p+1)+2\beta(q+1)>pq-1$
 			\item [(ii)] $r_\eps \to a$ \qquad \qquad for \ \ $2\al(p+1)+2\beta(q+1)< pq-1$
 			\item [(iii)] $\f{r_\eps^2}{2} \to \abs{x_0}$ \qquad for \ \ $2\al(p+1)+2\beta(q+1)= pq-1$
 		\end{itemize}
 		where $x_0$ maximizes the function $C_1H(x)\abs{x}^{\f{1}{2}}+C_2x_n\abs{x}^{-\f{3}{2}}$ on $\pa\Om$ for some constant $C_1$ and $C_2$.
 	\end{theorem}
 	
 	\begin{proof}
 		\par \smallskip \noindent  \\ 
 		{\it   Case (i)}: Under the assumption $2\al(p+1)+2\beta(q+1)>pq-1$ the minimum of $\La(x)$ occurs at the outer boundary of $\Om$, i.e at $r=\f{b^2}{2}$\\ 
 		{\it Case (ii)}: Under the assumption $2\al(p+1)+2\beta(q+1)<pq-1$ the minimum of $\La(x)$ occurs at the inner boundary of $\Om$, i.e at $r=\f{a^2}{2}$
 		\par \noindent \\
 		{\it   Case (iii)}: For the exponents satisfying $\al(q+1)+\beta(p+1)=pq-1$, note that $\La$ becomes constant and we go for the first order approximation to locate the concentration. Then, we get the point of concentration for the reduced problem \ef{appln5} maximizes $C_1H(x)\abs{x}^{\f{1}{2}}+C_2x_n\abs{x}^{-\f{3}{2}}$ in $\pa\Om$, where
 		\begin{align*}
 		C_1&=\f{5\sqrt{2}}{n+1}\int_{ \R^{n}_+ }<\grad U, \grad V>z_ndz\\
 		C_2&=\sqrt{2}\Big[\int_{ \R^{n}_+ }\Big[(\f{\al}{2}-1)F(U)+(\f{\beta}{2}-1)G(V)+UV\Big]z_ndz\Big].
 		\end{align*} Accordingly we get the circle of concentration as the orbit of this point.
 	\end{proof}

 	\section{Appendix A}

 	\subsection{Calculation of $\mu_j$ and $\nu_j$ :}
 	let $y=\eps_{j}x+x_0$ and $w=\Psi(y)$ and $z=\f{w}{\eps_j}$,
 	\begin{align*}
 		\f{\pa u_j(x)}{\pa x_i}=\Big[\f{\pa u}{\pa z_l}(z)\f{\pa \Psi_l}{\pa y_i}(y) \Big]\chi(w)+\eps_j u(z)\Big[\f{\pa \chi}{\pa w_l}(w)\f{\pa \Psi_l}{\pa y_i}(y)\Big]
 	\end{align*}
 	And hence
 	\begin{align*}
 		\f{\pa^2 u_j(x)}{\pa x_i^2}
 		=&\Big[\f{\pa^2 u}{\pa z_l\pa z_k}(z)\f{\pa \Psi_k}{\pa y_i}(y)\f{\pa \Psi_l}{\pa y_i}(y) +\eps_j\f{\pa u}{\pa z_l}(z)\f{\pa^2 \Psi_l}{\pa y_i^2}(y)
 		\Big]\chi(w)\\
 		&+2\eps_j\Big[\f{\pa u}{\pa z_l}(z)\f{\pa \Psi_l}{\pa y_i}(y) \Big]\Big[\f{\pa \chi}{\pa w_k}(w)\f{\pa \Psi_k}{\pa y_i}(y)\Big]\\
 		&+\eps_j^2 u(z)\Big[\f{\pa^2 \chi}{\pa w_l\pa z_k}(w)\f{\pa \Psi_k}{\pa y_i}(y)f{\pa \Psi_l}{\pa y_i}(y)+\f{\pa \chi}{\pa w_l}(w)\f{\pa^2 \Psi_l}{\pa y_i^2}(y)\Big]
 	\end{align*}
 	
 	So we get 
 	\begin{align}
 		\De u_j=& \Big[ \f{\pa^2 u}{\pa z_l\pa z_k}\f{\pa \Psi_k}{\pa y_i}\f{\pa \Psi_l}{\pa y_i}\Big]\chi(w)+\eps_j\Big[\f{\pa u}{\pa z_k}\f{\pa^{2}\Psi_k}{\pa^{2}z_i}\chi(w)+
 		2\f{\pa u}{\pa z_k}\f{\pa\Psi_k}{\pa y_i}\f{\pa\chi}{\pa w_l}\f{\pa \Psi_l}{\pa y_i}\Big]\notag\\
 		&+\eps^{2}_{j}u\Big[ \f{\pa^{2}\chi}{\pa w_k \pa w_l}\f{\pa \Psi_k}{\pa y_i}\f{\pa \Psi_l}{\pa y_i} + \f{\pa\chi}{\pa w_k}\f{\pa^{2}\Psi_k}{\pa y^{2}_i}\Big]\notag\\
 		=&(A_1+\eps_j A_2)\chi(w)+\eps_j A_3+\eps_j^2 A_4 \ (say) \label{decompuj}
 	\end{align}
 	
 	\begin{lemma}\label{APLemma1}
 		$A_1=\De u(z)+2\eps_j\sum_{k,l=1}^{n-1}\f{\pa^2 u}{\pa z_l\pa z_k} \f{\pa^2 \psi}{\pa x_k\pa x_l}(0)z_n+O(\abs{\eps_j z}^2)$
 	\end{lemma}
 	
 	\begin{proof}
 		From \cite{MR1219814, MR1115095} we get:\\
 		For $1\le k,l\le n-1$
 		\begin{align*}
 			&\sum_{i=1}^{n}\f{\pa \Psi_k}{\pa y_i}(y)\f{\pa \Psi_l}{\pa y_i}(y)\\
 			&=\sum_{i=1}^{n-1}\Big(\de_{ki}+ \f{\pa^2 \psi}{\pa x_i\pa x_k}(w')w_n\Big)\Big(\de_{li}+ \f{\pa^2 \psi}{\pa x_l\pa x_i}(w')w_n\Big)+\f{\pa \psi}{\pa x_k}\f{\pa \psi}{\pa x_l}(w')+O(\abs{w}^2)\\
 			&=\de_{kl}+2 \f{\pa^2 \psi}{\pa x_k\pa x_l}(w')w_n+O(\abs{w}^2)\\
 			&=\de_{kl}+2 \f{\pa^2 \psi}{\pa x_k\pa x_l}(0)w_n+O(\abs{w}^2)
 		\end{align*}

 		Similarly	for $1\le j\le n-1$ and $k=n$, using $\grad\psi(0)=0$, we get
 		\begin{align*}
 			&\sum_{i=1}^{n}\f{\pa \Psi_n}{\pa y_i}(y)\f{\pa \Psi_l}{\pa y_i}(y)
 			=O(\abs{w}^2)
 		\end{align*}
 		
 		and for $k=l=n$ we have
 		\begin{align*}
 			&\sum_{i=1}^{n}\Big(\f{\pa \Psi_n}{\pa y_i}(y)\Big)^2
 			=\sum_{i=1}^{n-1}\Big(-\f{\pa \psi}{\pa x_i}(w)\Big)^2+1+O(\abs{w}^2)
 			=1+O(\abs{w}^2)
 		\end{align*}
 		
 		Hence from \ref{DcyLe} we get 
 		\begin{align*}
 			A_1=\De u(z)+2\eps_j\sum_{k,l=1}^{n-1}\f{\pa^2 u}{\pa z_l\pa z_k} \f{\pa^2 \psi}{\pa x_k\pa x_l}(0)z_n+O(\abs{\eps_j z}^2e^{-\de\abs{z}})
 		\end{align*}
 		
 	\end{proof}	
 	
 	Using the fact $\De \Psi_k(0)=0 \text{ and } \De \Psi_n(0)=-\De\psi(0)$ we get
 	\begin{lemma}\label{APLemma2}
 		$A_2=-\f{\pa u}{\pa z_n}(z)\De\psi(0)+O(\abs{\eps_j z}e^{-\de\abs{z}})$
 	\end{lemma}

 	\begin{lemma}\label{APLemma3}
 		$u_j, \ v_j$ satisfies 
 		\begin{equation}
 		\left\{\begin{aligned}
 		&-\De u_j + c(x_0) u_j =  b(x_0) v_j^q+\mu_j(x), \\&  -\De  v_j + c(x_0) v_j  = a(x_0)u_j^p +\nu_j(x) \quad\text{  in } \Om_j\\
 		&u>0, \ v>0 \text{ in } \Om_j,\,\,\,\text{ and }\,\,\,\f{\pa u}{\pa\nu} = 0 = \f{\pa v}{\pa\nu} \text{ on }\pOm_j
 		\end{aligned}
 		\right.
 		\end{equation}
 		where (with Einstein summation, $1\leq k,l\leq (n-1)$)
 		\begin{align}
 			\mu_j(x)&=-2\eps_j \f{\pa^2 u}{\pa z_l\pa z_k}(z) \f{\pa^2 \psi}{\pa x_k\pa x_l}(0)z_n+\eps_j\f{\pa u}{\pa z_n}(z)\De\psi(0)+\eps_j^2O(\abs{z}^2e^{-\de\abs{z}})\label{MUj}\\
 			\nu_j(x)&=-2\eps_j \f{\pa^2 v}{\pa z_l\pa z_k}(z) \f{\pa^2 \psi}{\pa x_k\pa x_l}(0)z_n+\eps_j\f{\pa v}{\pa z_n}(z)\De\psi(0)+\eps_j^2O(\abs{z}^2e^{-\de\abs{z}})\label{NUj}
 		\end{align}
 	\end{lemma}
 	
 	\begin{proof}
 		Using the fact $u$ and $v$ decays exponentially, $1-\chi$ vanishes in $B(0,\de)$ we get the result easily from \ref{decompuj}, Lemma\ref{APLemma1} and Lemma\ref{APLemma2}.
 	\end{proof}
 	
 	\begin{lemma}\label{APLemma4}
 		\begin{align*} 
 			\xi \De \psi(0)+O(\eps_j)&=\sum_{k,l=1}^{n-1}\f{\pa^2 \psi}{\pa x_k\pa x_l}(0)\int_{\Om_j} \f{\pa^2 u}{\pa z_l\pa z_k}z_nv_j(z)dx\\
 			&=\sum_{k,l=1}^{n-1}\f{\pa^2 \psi}{\pa x_k\pa x_l}(0)\int_{\Om_j} \f{\pa^2 v}{\pa z_l\pa z_k}z_nu _j(z)dx
 		\end{align*}
 		where $\xi=\int_{\rn_+} \f{\pa u}{\pa z_k}(z) \f{\pa v}{\pa z_k}(z)z_n dz$ for $1\le k\le n-1$
 	\end{lemma}
 	\begin{proof}
 		Using the exponential decay estimate \ref{DcyLe} we get
 		\begin{align*}
 			&\int_{\Om_j}\f{\pa^2 u}{\pa z_l\pa z_k}(z)z_nv_j(z)dx\\
 			&=\int_{\Om_j}\f{\pa^2 u}{\pa z_l\pa z_k}\Big(\f{\Psi(\eps_j x+x_0)}{\eps_j}\Big)\f{\Psi_n(\eps_j x+x_0)}{\eps_j}v\Big(\f{\Psi(\eps_j x+x_0)}{\eps_j}\Big)\chi(\Psi(\eps_j x+x_0))dx\\
 			&=\eps_j^{-n}\int_{\Om}\f{\pa^2 u}{\pa z_l\pa z_k}\Big(\f{\Psi(y)}{\eps_j}\Big)\f{\Psi_n(y)}{\eps_j}v\Big(\f{\Psi(y)}{\eps_j}\Big)\chi(\Psi(y))dy\\
 			&=\eps_j^{-n}\int_{\rn_+}\f{\pa^2 u}{\pa z_l\pa z_k}\Big(\f{w}{\eps_j}\Big)\f{w_n}{\eps_j}v\Big(\f{w}{\eps_j}\Big)\chi(w)\abs{det(D\Phi(w))}dw\\
 			&=\int_{\rn_+}\f{\pa^2 u}{\pa z_l\pa z_k}(z)z_n v(z)\chi(\eps_j z)\abs{det(D\Phi(\eps_j z))}dz\\
 			&=\int_{\rn_+}\f{\pa^2 u}{\pa z_l\pa z_k}(z)z_n v(z)\chi(\eps_j z)(1-(n-1)H(0)\eps_j z_n+O(\abs{\eps_j z}^2))dz\\
 			&=\int_{\rn_+}\f{\pa^2 u}{\pa z_l\pa z_k}(z)z_n v(z)(1-(n-1)H(0)\eps_j z_n+O(\abs{\eps_j z}^2))dz\\
 			&\,\, -\int_{\rn_+\backslash B(0,\de/\eps_j)}\f{\pa^2 u}{\pa z_l\pa z_k}(z)z_n v(z)(1-(n-1)H(0)\eps_j z_n+O(\abs{\eps_j z}^2))dz\\
 			&\,\, +\int_{\rn_+\backslash B(0,\de/\eps_j)}\f{\pa^2 u}{\pa z_l\pa z_k}(z)z_n v(z)\chi(\eps_j z)(1-(n-1)H(0)\eps_j z_n+O(\abs{\eps_j z}^2))dz\\
 			&=\int_{\rn_+}\f{\pa^2 u}{\pa z_l\pa z_k}(z)z_n v(z)(1-\eps_j(n-1)H(0) z_n))dz+o(\eps_j)
 		\end{align*}
 		\\
 		
 		Now  for $1\le k, l\le n-1$ we have
 		
 		\begin{align*}
 			\int_{\rn_+}\f{\pa^2 u}{\pa z_l\pa z_k}(z)z_n v(z)dz
 			=\int_{\rn_+}\f{\pa u}{\pa z_l}(z) \f{\pa v}{\pa z_k}(z) z_n dz
 		\end{align*}
 		
 		hence using the symmetry in the integral we get 
 		
 		\begin{align*}
 			&\sum_{k,l=1}^{n-1}\f{\pa^2 \psi}{\pa x_k\pa x_l}(0)\int_{\Om_j} \f{\pa^2 u}{\pa z_l\pa z_k}z_nv_j(z)dx\\
 			&= \sum_{k=1}^{n-1}\f{\pa^2 \psi}{\pa x_k\pa x_l}(0)\int_{\rn_+} \f{\pa u}{\pa z_k}(z) \f{\pa v}{\pa z_k}(z)z_n dz\\
 			&= \xi \De \psi(0)\\
 		\end{align*}
 		Similar calculation shows the other equality.
 	\end{proof}
 	
 	\begin{lemma}\label{APLemma5}
 		\be\int_{ \R^{n}_+ }\Big[\f{\pa u}{\pa z_n} v+u\f{\pa v}{\pa z_n}\Big]dz =\int_{\pa\rn_+}uvd\sigma\nee
 	\end{lemma}
 	\begin{proof}
 		Proof is obvious from integration by parts formula.
 	\end{proof}

 	\section{Appendix B}

 	\subsection{Calculation of $\bar{\mu}_j$ and $\bar{\nu}_j$ :}
 	let $w=\eps_{j}z$ and $y=\Phi(w)$ and $x=\f{y-x_j}{\eps_j}$,
 	\begin{align*}
 		\f{\pa \bar{u}_j(z)}{\pa z_i}&=\Big[\f{\pa \tilde{u}}{\pa x_l}(x)\f{\pa \Phi_l}{\pa y_i}(w) \Big]\chi(\eps_j x)+\eps_j \tilde{u}(x)\Big[\f{\pa \chi}{\pa w_l}(\eps_j x)\f{\pa \Phi_l}{\pa y_i}(w)\Big]
 	\end{align*}
 	Then 
 	\begin{align*}
 		\f{\pa^2 \bar{u}_j(z)}{\pa x_i^2}=&\Big[\f{\pa^2 \tilde{u}}{\pa x_l\pa x_k}(x)\f{\pa \Phi_k}{\pa w_i}(w)\f{\pa \Phi_l}{\pa w_i}(w) +\eps_j\f{\pa \tilde{u}}{\pa x_l}(x)\f{\pa^2 \Phi_l}{\pa w_i^2}(w)
 		\Big]\chi(\eps_j x)\\
 		&+2\eps_j\Big[\f{\pa \tilde{u}}{\pa x_l}(x)\f{\pa \Phi_l}{\pa y_i}(y) \Big]\Big[\f{\pa \chi}{\pa w_k}(w)\f{\pa \Phi_k}{\pa y_i}(y)\Big]\\
 		&+\eps_j^2 \tilde{u}(x)\Big[\f{\pa^2 \chi}{\pa x_l\pa x_k}(\eps_j x)\f{\pa \Phi_k}{\pa w_i}(w)\f{\pa \Phi_l}{\pa w_i}(w)+\f{\pa \chi}{\pa x_l}(\eps_j x)\f{\pa^2 \Phi_l}{\pa w_i^2}(w)\Big]
 	\end{align*}

 	Hence 
 	\begin{align}
 		\De \bar{u}_j=& \Big[ \f{\pa^2 \tilde{u}}{\pa x_l\pa x_k}\f{\pa \Phi_k}{\pa w_i}\f{\pa \Phi_l}{\pa w_i}\Big]\chi(\eps_j x)+\eps^{2}_{j}\tilde{u}\Big[ \f{\pa^{2}\chi}{\pa x_k \pa x_l}\f{\pa \Phi_k}{\pa w_i}\f{\pa \Phi_l}{\pa w_i} + \f{\pa\chi}{\pa x_k}\f{\pa^{2}\Phi_k}{\pa w^{2}_i}\Big]\notag\\
 		&+\eps_j\Big[\f{\pa \tilde{u}}{\pa x_k}\f{\pa^{2}\Phi_k}{\pa^{2}w_i}\chi(\eps_j x)+
 		2\f{\pa \tilde{u}}{\pa x_k}\f{\pa\Phi_k}{\pa w_i}\f{\pa\chi}{\pa x_l}\f{\pa \Phi_l}{\pa w_i}\Big]\notag\\
 		=&(A_1+\eps_j A_2)\chi(w)+\eps_j A_3+\eps_j^2 A_4 \ (say) \label{decompujB}
 	\end{align}
 	
 	\begin{lemma}\label{APLemma1B}
 		$A_1=\De \tilde{u}_j(x)-2\eps_j\sum_{k,l=1}^{n-1}\f{\pa^2 \tilde{u}_j}{\pa x_l\pa x_k} \f{\pa^2 \psi}{\pa z_k\pa z_l}(0)z_n+O(\abs{\eps_j z}^2)$
 	\end{lemma}
 	
 	\begin{proof}
 		From \cite{MR1219814, MR1115095} we get:\\
 		For $1\le k,l\le n-1$
 		\begin{align*}
 			&\sum_{i=1}^{n}\f{\pa \Phi_k}{\pa y_i}(y)\f{\pa \Phi_l}{\pa y_i}(y)\\
 			&=\sum_{i=1}^{n-1}\Big(\de_{ki}- \f{\pa^2 \psi}{\pa x_i\pa x_k}(y')y_n\Big)\Big(\de_{li}- \f{\pa^2 \psi}{\pa x_l\pa x_i}(y')y_n\Big)+\f{\pa \psi}{\pa x_k}(y')\f{\pa \psi}{\pa x_l}(y')+O(\abs{y}^2)\\
 			&=\de_{kl}-2 \f{\pa^2 \psi}{\pa x_k\pa x_l}(y')y_n+O(\abs{y}^2)\\
 			&=\de_{kl}-2 \f{\pa^2 \psi}{\pa x_k\pa x_l}(0)y_n+O(\abs{y}^2)
 		\end{align*}
 		
 		Similarly for $1\le j\le n-1$ and $k=n$, using $\grad\psi(0)=0$, we get
 		\begin{align*}
 			&\sum_{i=1}^{n}\f{\pa \Phi_n}{\pa y_i}(y)\f{\pa \Phi_l}{\pa y_i}(y)=O(\abs{y}^2)
 		\end{align*}
 		
 		and for $k=l=n$ we have
 		
 		\begin{align*}
 			\sum_{i=1}^{n}\Big(\f{\pa \Phi_n}{\pa y_i}(y)\Big)^2=\sum_{i=1}^{n-1}\Big(\f{\pa \psi}{\pa x_i}(y)\Big)^2+1+O(\abs{y}^2)=1+O(\abs{y}^2)
 		\end{align*}

 		Hence from \ref{DcyLe} we get 
 		\begin{align*}
 			A_1=\De \tilde{u}_j(x)-2\eps_j\sum_{k,l=1}^{n-1}\f{\pa^2 \tilde{u}_j}{\pa x_l\pa x_k} \f{\pa^2 \psi}{\pa z_k\pa z_l}(0)z_n+O(\abs{\eps_j z}^2e^{-\theta\abs{z}})
 		\end{align*}
 		
 	\end{proof}

 	\begin{lemma}\label{APLemma2B}
 		$A_2=\f{\pa \tilde{u}_j}{\pa z_n}(z)\De\psi(0)+O(\abs{\eps_j z}e^{-\de\abs{z}})$
 	\end{lemma}
 	
 	\begin{proof}
 		
 		\begin{align*}
 			\De \Phi_k(0)=0 \text{ and } \De \Phi_n(0)=\De\psi(0).
 		\end{align*}A direct calculation shows that:
 		  $A_2=\f{\pa \tilde{u}_j}{\pa z_n}(z)\De\psi(0)+O(\abs{\eps_j z}e^{-\de\abs{z}})$
 		
 	\end{proof}
 	
 	\begin{lemma}\label{APLemma3B}
 		$u_j, \ v_j$ satisfies 
 		\begin{equation}
 		\left\{\begin{aligned}
 		&	-\De \bar{u}_j + c(\eps_j x+x_j) \tilde{u}_j =  b(\eps_j x+x_j) \tilde{v}_j^q+\bar{\mu}_j(z), \quad \text{in } \Om_j\\
 		& -\De  \bar{v}_j + c(\eps_j x+x_j) \tilde{v}_j  = a(\eps_j x+x_j)\tilde{u}_j^p +\bar{\nu}_j(z)\quad \text{in } \Om_j\\
 		&u>0, \ v>0 \text{ in } \Om_j, \text{ and }\,\,\f{\pa u}{\pa\nu} = 0 = \f{\pa v}{\pa\nu} \text{ on }\pOm_j
 		\end{aligned}
 		\right.
 		\end{equation}
 		where  (with Einstein summation, $1\leq k,l\leq (n-1)$)
 		\begin{align}
 			\bar{\mu}_j(z)&=2\eps_j \f{\pa^2  \tilde{u}}{\pa x_l\pa x_k}(x) \f{\pa^2 \psi}{\pa z_k\pa z_l}(0)z_n-\eps_j\f{\pa  \tilde{u}}{\pa x_n}(x)\De\psi(0)+\eps_j^2O(\abs{z}^2e^{-\theta\abs{z}})\label{MUjB}\\
 			\bar{\nu}_j(z)&=2\eps_j \f{\pa^2 \tilde{v}}{\pa x_l\pa x_k}(x) \f{\pa^2 \psi}{\pa x_k\pa x_l}(0)z_n-\eps_j\f{\pa \tilde{v}}{\pa x_n}(x)\De\psi(0)+\eps_j^2O(\abs{z}^2e^{-\theta\abs{z}})\label{NUjB}
 		\end{align}
 	\end{lemma}
 	
 	\begin{proof}
 		Using the fact $\tilde{u}$ and $\tilde{v}$ decays exponentially, $1-\chi$ vanishes in $B(0,\de)$ we get the result. 
 	\end{proof}
 	
 	Also using the $C^2_{loc}$ convergence and by dominated convergence theorem we get the following 
 	
 	\begin{lemma}\label{APLemma4B}
 		\begin{align*} 
 			\xi \De \psi(0)+O(\eps_j)&=\sum_{k,l=1}^{n-1}\f{\pa^2 \psi}{\pa x_k\pa x_l}(0)\int_{\rn_+} \f{\pa^2 \tilde{u}}{\pa x_l\pa x_k}(x)z_n\bar{v}_j(z)dz\\
 			&=\sum_{k,l=1}^{n-1}\f{\pa^2 \psi}{\pa x_k\pa x_l}(0)\int_{\rn_+} \f{\pa^2 \tilde{v}}{\pa x_l\pa x_k}(x)z_n\bar{u}_j(z)dz
 		\end{align*}
 		where $\xi=\int_{\rn_+} \f{\pa u}{\pa z_k}(z) \f{\pa v}{\pa z_k}(z)z_n dz$ for $1\le k\le n-1$
 	\end{lemma}

 	\begin{lemma}\label{APLemma5B}
 		\be\int_{ \R^{n}_+ }\Big[\f{\pa \tilde{u}_j}{\pa x_n}(x) \bar{v}_j(z)+\bar{u}_j(z)\f{\pa \tilde{v}_j}{\pa x_n}\Big]dz =\int_{\pa\rn_+}uvd\sigma+O(\eps_j)\nee
 	\end{lemma}
 	\begin{proof}
 		Proof is obvious from integration by parts formula.
 	\end{proof}

 	\section*{Acknowledgment}
 	
 	A.K. Sahoo is supported by fellowship under ministry of human resource development(India).


\begin{thebibliography}{15}
 		\bibitem{MR1804760}   Abbondandolo, A.; Felmer, P.; Molina, J. An estimate on the relative Morse index for strongly indefinite functionals.  Electron. J. Diff. Eqns., Conf. 06 (2001), 1–11.
 		
 		\bibitem{MR1974510}
 		Antonio Ambrosetti, Andrea Malchiodi, and Wei-Ming Ni.
 		\newblock Singularly perturbed elliptic equations with symmetry: existence of
 		solutions concentrating on spheres. {I}.
 		\newblock {  Comm. Math. Phys.}, \textbf{235}(3):427--466, 2003.
 		
 		\bibitem{MR2056434}
 		Ambrosetti, A.,Malchiodi, A., Ni,W.-M.:
 		\newblock Singularly perturbed elliptic equations with symmetry: existence of
 		solutions concentrating on spheres. {II}.
 		\newblock {  Indiana Univ. Math. J.} \textbf{53}(2), 297--329 (2004)
 		
 		
 		\bibitem{MR2057542}
 		Angela Pistoia and Miguel Ramos.
 		\newblock Locating the peaks of the least energy solutions to an elliptic
 		system with {N}eumann boundary conditions.
 		\newblock {  J. Differential Equations}, \textbf{201}(1):160--176, 2004.
 		
 		
 	 
 		
 		
 		\bibitem{MR1978382}
 		\'{A}vila, Andrés I.; Yang, Jianfu.
 		\newblock On the existence and shape of least energy solutions for some elliptic systems.
 		\newblock {   J. Differential Equations.},\textbf{191}(2), 348–-376,2003
 		
 		
 		\bibitem{MR2608946}
 		Bernhard Ruf and P.~N. Srikanth.
 		Singularly perturbed elliptic equations with solutions concentrating
 		on a 1-dimensional orbit.
 		J. Eur. Math. Soc. (JEMS), \textbf{12}(2):413--427, 2010.
 		
 		
 		
 		\bibitem{MR3250368}
 		Bernhard Ruf and P.~N. Srikanth.
 		Concentration on {H}opf-fibres for singularly perturbed elliptic
 		equations.
 		J. Funct. Anal., \textbf{267}(7):2353--2370, 2014.
 		
 		
 		
 		\bibitem{MR3056708}
 		Bonheure, Denis and Serra, Enrico and Tilli, Paolo:
 		\newblock Radial positive solutions of elliptic systems with {N}eumann
 		boundary conditions.
 		\newblock {  J. Funct. Anal.} \textbf{265}(3), 375--398 (2013)
 		
 			\bibitem{MR2180862}
 		Byeon, Jaeyoung; Park, Junsang Singularly perturbed nonlinear elliptic problems on manifolds. \newblock {\em Calc. Var. Partial Differential Equations} 24 (2005), no. 4, 459--477. 
 		
 		\bibitem{MR1177298}
 		Cl\'{e}ment, Ph.; de Figueiredo, D. G.; Mitidieri, E.
 		\newblock Positive solutions of semilinear elliptic systems.
 		\newblock{  Comm. Partial Differential Equations} \textbf{17} (1992), no. 5-6, 923–-940. 
 		
 		
 	 
 		
 		
 		 
 		
 		
 		\bibitem{MR1617988}
 		Djairo~G. De~Figueiredo and Jianfu Yang.
 		\newblock Decay, symmetry and existence of solutions of semilinear elliptic systems.
 		Nonlinear Anal, \textbf{33}(3):211--234, 1998.
 		
 		
 		\bibitem{MR3192458}
 		Filomena Pacella and P.~N. Srikanth.
 		A reduction method for semilinear elliptic equations and solutions
 		concentrating on spheres.
 		J. Funct. Anal, \textbf{266}(11):6456--6472, 2014.
 		
 		
 	 
 	 
 		
 	 
 		
 		
 		\bibitem{MR3259003}
 		Manna, Bhakti B.; Srikanth, P. Chakravarthy 
 		On the solutions of a singular elliptic equation concentrating on a circle.   Adv. Nonlinear Anal.  \textbf{3} (2014), 141–155.
 		
 		\bibitem{MR2342612}
 		Miguel Ramos and Hugo Tavares.
 		Solutions with multiple spike patterns for an elliptic system.
 		Calc. Var. Partial Differential Equations , \textbf{31}(1):1--25, 2008.
 		
 		\bibitem{MR2135746}
 		Miguel Ramos and Jianfu Yang.
 		Spike-layered solutions for an elliptic system with {N}eumann
 		boundary conditions.
 		Trans. Amer. Math. Soc, \textbf{357}(8):3265--3284 (electronic), 2005.
 		
 		
 	 
 		
 		
 		\bibitem{MR1219814}	  
 		Ni, Wei-Ming; Takagi, Izumi
 		Locating the peaks of least-energy solutions to a semilinear Neumann problem.
 		Duke Math. J. \textbf{70} (1993), no. 2, 247-–281
 		
 		
 	 
 		
 	 
 		
 		
 		\bibitem{MR1115095}
 		Wei-Ming Ni and Izumi Takagi.
 		On the shape of least-energy solutions to a semilinear {N}eumann
 		problem.
 		Comm. Pure Appl. Math., \textbf{44}(7):819--851, 1991. 	
 		
 		
 		\bibitem{MR2250232}
 		John~C. Wood.
 		Harmonic morphisms between {R}iemannian manifolds.
 		In  Modern trends in geometry and topology, pages 397--414. Cluj
 		Univ. Press, Cluj-Napoca, 2006.
 		
 		\bibitem{MR1785681}   Sirakov, Boyan On the existence of solutions of Hamiltonian elliptic systems in $R^N$. Adv. Differential Equations 5 (2000), no. 10-12, 1445--1464.
 		
 		
 		\bibitem{MR3229826}
 		M\'{o}nica Clapp, Marco Ghimenti, and Anna~Maria Micheletti.
 		\newblock Solutions to a singularly perturbed supercritical elliptic equation
 		on a {R}iemannian manifold concentrating at a submanifold.
 		\newblock {\em J. Math. Anal. Appl.}, 420(1):314--333, 2014.
 		
 		\bibitem{MR3620895}
 		M\'{o}nica Clapp and Bhakti~Bhusan Manna.
 		\newblock Double- and single-layered sign-changing solutions to a singularly
 		perturbed elliptic equation concentrating at a single sphere.
 		\newblock {\em Comm. Partial Differential Equations}, 42(3):474--490, 2017.
 		
 		\bibitem{MR3032301}
 		Nils Ackermann, M{\'o}nica Clapp, and Angela Pistoia.
 		\newblock Boundary clustered layers near the higher critical exponents.
 		\newblock {\em J. Differential Equations}, 254(10):4168--4193, 2013.
 		
 
 		
 		
 	\end{thebibliography}
 \end{document}